\documentclass{amsart}

\newtheorem{theorem}{Theorem}

\newtheorem{proposition}{Proposition}
\newtheorem{corollary}{Corollary}
\newtheorem{remark}{Remark}
\newtheorem{definition}{Definition}
\newtheorem{lemma}{Lemma}

\newcommand{\Z}{{\mathbb Z}}

\newcommand{\R}{{\mathbb R}}
\newcommand{\C}{{\mathbb C}}

\newcommand{\RP}{{\mathbb RP}}

\begin{document}

\title[Complements of discriminants of simple singularities]{Complements of  discriminants of  simple real function singularities}
\author{V.A.~Vassiliev}
\email{vva@mi-ras.ru}
\address{Steklov Mathematical Institute of Russian Academy of Sciences \ \ and \newline Higher School of Economics, Moscow, Russia}
%

%\dedication{A dedication can be included here.}
%\thanks{This work was performed at the Steklov International Mathematical Center and supported by the Ministry of Science and %Higher Education of the Russian Federation (agreement no. 075-15-2019-1614  } 

\begin{abstract}
All components of complements of discriminant varieties of simple real function singularities are explicitly listed. New invariants of such components (for not necessarily simple singularities) are introduced. A combinatorial algorithm enumerating topological types of morsifications of real function singularities is promoted.
\end{abstract}

\maketitle

\section{Introduction}

Let $f:(\R^n,0) \to (\R,0)$, $df(0)=0$, be a smooth function singularity, and $F:(\R^n \times \R^l,0) \to (\R,0)$ be its arbitrary smooth deformation, i.e. a family of functions $f_\lambda \equiv F(\cdot, \lambda): \R^n \to \R,$ $f_0 \equiv f$. {\it Real discriminant} $\Sigma = \Sigma(F)$ of this deformation is the set of parameters $\lambda \in \R^l$ such that function $f_\lambda$ has a critical point   in $\R^n$ with zero critical value. This is a subvariety   in $\R^l$ (of codimension 1 in all interesting cases),  which can divide a neighborhood of the origin  into several connected components.  Discriminants appear  in many problems of PDE theory, mathematical physics and integral geometry as singular loci (aka {\it wavefronts}) of related special functions such as fundamental solutions of certain PDEs. The behavior of these functions in different connected components of the complements of discriminant varieties (where they are regular) can be very different; see e.g. \cite{petr}, \cite{Leray}, \cite{Her}, \cite{ABG}, \cite{Gording}, \cite{APLT}, \cite{Var}. Thus, the problem of the enumeration of these components arises at an early step of the study of these functions. 

Suppose that the  {\em Milnor number} of singularity $f$ is finite (see e.g. \cite{AVGZ} which is a standard reference on function singularities and their deformations). It is then natural to study the discriminant subvariety of an arbitrary {\em versal} deformation of $f$, to any of which all other deformations can be reduced. In particular, there is a one-to-one correspondence between the components of complements of discriminants of any two versal deformations of our singularity $f$. The shape and quantity of these components depend only on the singularity type of $f$ at the origin. The first natural class of such types consists of the so-called {\em simple singularities}. 
  Any germ of a simple singularity in $n\geq 2$ variables can be reduced by a choice of local coordinates to the form $f(x, y)+Q(z_1, \dots, z_{n-2})$, where $f$ is one of the polynomial normal forms listed in Table \ref{t1} and $Q$ is a non-degenerate quadratic form.  Milnor numbers of these singularities are indicated by lower indices in  notation $A_\mu$, $D_{2k}^-$, etc.  of these types.
\begin{table}
\begin{center}
\begin{tabular}{|l|l|c|}
\hline
Notation & Normal form &  \# of components   \\ 
\hline 
$A_\mu$, \quad $\mu \ge 1$ & $\pm x^{\mu+1} \pm y^2 $ &  $\left[\frac{\mu+3}{2}\right]$\\ [3pt]
$D_{2k}^-$;\quad $k \ge 2$ & $x^2y \/- y^{2k-1} $ &  $\frac{(k+2)(k+1)}{2}+1$   \\ [3pt]
$D_{2k}^+$;\quad $k \ge 2$ & $x^2y \/+ y^{2k-1} $ &  $\frac{(k+1)k}{2}$   \\ [4pt]
$\pm D_{2k-1}$; \ $k \ge 3$ & $\pm (x^2y \/+ y^{2k})$  & $\frac{(k+1)k}{2}$  \\ [3pt]
$\pm E_6$ & $x^3 \pm y^4 $ & 5 \cr
$E_7$ & $x^3 + xy^3 $ & 10\cr
$E_8$ & $x^3 + y^5$ & 10 \cr
\hline \end{tabular} \end{center}
\caption{Normal forms of real simple singularities in two variables}
\label{t1}
\end{table} 
A versal deformation of a simple singularity with Milnor number $\mu$ can be chosen in the form $F(x, y ;\lambda) + Q(z_1, \dots, z_{n-2})$, $\lambda =(\lambda_1, \dots, \lambda_\mu) \in \R^\mu$, where deformation $F(x, y; \lambda)$ is given by the corresponding formula from the following list: 
\begin{eqnarray}
A_\mu &  & f + \lambda_1 + \lambda_2x + \lambda_3x^2 + \dots + \lambda_{\mu} x^{\mu-1} \label{avd} \\
D_\mu &  & f + \lambda_1 + \lambda_2x + \lambda_3y+
\lambda_4y^2 + \cdots + \lambda_{\mu}y^{\mu -2} \label{dvd} \\
 E_6 & & f + \lambda_1 + \lambda_2x + \lambda_3y +
\lambda_4xy + \lambda_5y^2 + \lambda_6 xy^2 \label{e6vd} \\
 E_7 &  & f + \lambda_1 + \lambda_2x + \lambda_3y +
\lambda_4xy + \lambda_5y^2 + \lambda_6y^3 + \lambda_7y^4 \label{e7vd} \\ 
E_8 &  & f + \lambda_1 + \lambda_2x + \lambda_3y +
\lambda_4xy + \lambda_5y^2 + \lambda_6xy^2 +
\lambda_7y^3 + \lambda_8xy^3 \label{e8vd} 
\end{eqnarray}

Such deformations with equal $F$ but different $Q$ and $n$  have equal discriminant varieties, therefore it is enough to study them in case $n=2$ only. 

E.~Looijenga \cite{Lo} has proved a one-to-one correspondence between the components of the complement of the discriminant variety of a simple singularity and certain algebraic objects related to the corresponding reflection group, see \S \ref{woL} below. We use a different approach (sometimes using the Looijenga's results in the justifications) and present  explicit lists of these components, in particular give general formulas for numbers of them for serial singularity types.

\begin{theorem}
\label{mthm}
The numbers of local components of complements of discriminant varieties of  versal deformations of real simple singularities are as indicated in the right-hand column of Table \ref{t1}.
\end{theorem}

\begin{remark} \rm
The combinatorics of discriminants of singularities with Milnor numbers up to 6 was explicitly studied by V.~Sedykh \cite{sedykh}, see also \cite{sedykh2}.
In particular, these numbers for singularities $D_4$, $D_5$, $D_6$, and $E_6$ are given in Theorems 2.8 and 2.9 of \cite{sedykh}.
\end{remark}

In case $A_\mu$ the enumeration problem is trivial. Indeed, it can be reduced to the study of polynomials of degree $\mu+1$ in one variable. Such a polynomial is non-discriminant if and only if all its real roots are simple; the number of these real roots is a complete invariant of the component of the set of non-discriminant polynomials. 

In all cases we will assume that  function $f$ has  one of polynomial normal forms from Table \ref{t1} (in particular depends on two variables), and its versal deformation $F$ has the standard form given by corresponding formula (\ref{avd})--(\ref{e8vd}). 

For any parameter value $\lambda \in \R^\mu$ of either of these deformations,  consider the set of lower values 
 of the corresponding function, $W(\lambda) = \{X \in \R^2|f_\lambda(X) \leq 0\}$. These sets go to infinity in $\R^2$ along several {\em asymptotic sectors} not depending on $\lambda$ and corresponding to the domains of constant sign of the common principal  part $f$ of all functions $f_\lambda$ of this deformation. The number of these sectors is equal to 3 for $D_{2k}^-$, to 2 for $A_{2k-1}$ (the case of different signs $\pm$ in the normal form of Table \ref{t1}) or $\pm D_{2k-1}$ or $E_7$; to 0 for $A_{2k-1}$ (the case of both signs $+$), and to 1 in remaining cases. 

\begin{definition} \rm
Two sets of lower values, $W(\lambda)$ and $W(\lambda')$, corresponding to points $\lambda$ and $\lambda'$ of parameter space $\R^\mu$ of one of the deformations (\ref{avd})--(\ref{e8vd}) are {\em topologically equivalent} if they can be moved one to the other by an orientation preserving diffeomorphism of $\R^2$ not permuting asymptotic sectors. Topological equivalence classes of sets $W(\lambda)$ will be called also their {\em topological types}.
\end{definition}

It is easy to see that if $\lambda$ and $\lambda'$ belong to one and the same component of the complement of the discriminant variety then corresponding sets $W(\lambda)$ and $W(\lambda')$ are topologically equivalent. We prove below that the converse is also true.

\begin{theorem}
\label{mthm2}
If $\lambda$ and $\lambda'$ are non-discriminant points of parameter space $\R^\mu$ of a versal deformation from the list $($\ref{avd}$)$--$($\ref{e8vd}$)$, and corresponding sets  $W(\lambda), W(\lambda')$ are topologically equivalent, then $\lambda$ and $\lambda'$ belong to the same component of $\R^\mu \setminus \Sigma$.
\end{theorem}

All possible topological types of sets $W(\lambda)$ for simple singularities will be listed in \S \ref{types}. In \S \ref{realiz} we prove that these types can be realized by certain perturbations $f_\lambda$ of these singularities. The fact that each of these types corresponds to only one connected component of $\R^\mu \setminus \Sigma$ and there are no other topological types will be proved in \S \ref{esti} for singularities $D_\mu$ and in \S \ref{estie} for $E_\mu$. 

Our lists were guessed (and in the case of $E_\mu$ singularities partly justified) with the help of a combinatorial computer program enumerating  topologically distinct morsifications of (not necessarily simple) function singularities, see \S \ref{adad}.

\begin{remark} \rm
Enumeration of complements of discriminants can be considered as a ramification of the problem of real algebraic geometry on the rigid isotopy of real algebraic varieties (see e.g. \cite{rohlin}, \cite{kharlamov},  \cite{pol},  \cite{viro}),  namely of its affine version with some boundary constraints and singularities at infinity. In particular, the study of $E_\mu$ singularities deals only with curves of degree $\leq 5$, the classification of which (at least for smooth projective curves and for smooth affine curves generic at the infinity) has been done for decades. However, our problem is not an a priori  subproblem of such a classification, because a rigid isotopy through the space of all curves of a certain degree is not necessarily reducible to a path within the parameter space of a deformation considered in our case. 
\end{remark}

\section{E.~Looijenga's results}
\label{woL}

\begin{theorem}[see \cite{Lo}] 
\label{lotriv}
All components of the complement of the real discriminant variety of any of versal deformations $($\ref{avd}$)$--$($\ref{e8vd}$)$ are contractible.
\end{theorem}

All functions $f$ from Table \ref{t1} and their deformations (\ref{avd})--(\ref{e8vd}) can be extended to the complex domain and be considered as functions $\C^2 \to \C$ and $\C^2 \times \C^\mu \to \C$.

For any parameter value $\lambda \in \C^\mu $ of one of complexified deformations  (\ref{avd})--(\ref{e8vd})
consider the polynomial $f_\lambda(x, y)+z^2:\C^3 \to \C$. The corresponding {\em Milnor fiber} $V_\lambda \subset \C^3$ is defined by the equation $f_\lambda(x, y) + z^2 =0$. The {\it complex discriminant} $\Sigma_{\C}$ is the set of values $\lambda \in \C^\mu$ such that $V_\lambda$ is singular. The variety $\Sigma_\C \cap \R^\mu$ contains the real discriminant considered above, but in non-trivial cases is  greater  because a polynomial $f_\lambda$ with real coefficients can have imaginary critical points with critical value 0. However, the difference $(\Sigma_{\C} \cap \R^\mu) \setminus \Sigma$ has codimension $\geq 2$ in $\R^\mu$, therefore there is a one-to-one correspondence between the components of complements of these varieties.

For any $\lambda \in \C^\mu \setminus\Sigma_{\C}$, the Milnor fiber $V_\lambda$ is an oriented 4-dimensional manifold homotopy equivalent to the wedge of $\mu$ two-dimensional spheres, in particular $H_2(V_\lambda, \Z) \simeq \Z^\mu$ (see \cite{Milnor}). Intersection indices of cycles define a bilinear form  $\langle \cdot, \cdot \rangle$   on this lattice. For all simple singularities the corresponding quadratic form is negative definite. The {\em monodromy group}  $M_\lambda$ of our singularity with basepoint $\lambda$, i.e. the natural representation \ $\pi_1(\C^\mu \setminus \Sigma_\C, \lambda) \to \mbox{Aut} (H_2(V_\lambda, \Z))$ (see e.g. \cite{AVGZ2}, \cite{APLT}), preserves this form and hence is a subgroup  of  group $\mbox{O}( H_2(V_\lambda, \Z), \langle \cdot, \cdot \rangle)$. If $\lambda \in \R^\mu \setminus \Sigma_\C$ then the latter group contains also element $\sigma_\lambda$ induced by the complex conjugation of $V_\lambda$.

The lattices $H_2(V_\lambda, \Z)$ for different non-discriminant points $\lambda$ can be identified to one another by isomorphisms defined by paths connecting these points in $\C^\mu \setminus \Sigma_\C$. However, such isomorphisms depend on the paths, so that two such isomorphisms can differ by an element of the group $M_\lambda$. 
The sets of conjugacy classes of the group $\mbox{O}( H_2(V_\lambda, \Z), \langle \cdot, \cdot \rangle)$ by  $M_\lambda$ for different $\lambda$ are in a canonical one-to one correspondence with one another. 
Let $N_\lambda$ be the normalizer of the subgroup $M_\lambda$
in $\mbox{O}( H_2(V_\lambda, \Z), \langle \cdot, \cdot \rangle)$. The quotient groups
$N_\lambda/M_\lambda$ for all non-discriminant $\lambda$ are canonically isomorphic to one another; denote the resulting object by $N/M$.

\begin{proposition}[see \cite{Lo}]
\label{1.7.2.prop}
1. The quotient group $N/M$ is canonically isomorphic to the group
of automorphisms of the standard Dynkin graph of the considered simple
singularity $($i.e. is trivial for 
$A_1, E_7, E_8$, is isomorphic to the group $S(3)$ for $D_4$ and to
the group $\Z_2$ in the remaining cases$)$.

2. For any $\lambda \in \R^\mu \setminus \Sigma_{\C}$  action $\sigma_\lambda$ of the complex conjugation on
$H_2(V_\lambda, \Z)$ belongs to group $N_\lambda$. Cosets $\{M_\lambda\sigma_\lambda\} \in N_\lambda/M_\lambda$ define one and the same element of  group $N/M$ for all
$\lambda \in \R^\mu \setminus \Sigma_\C$.
In particular, class $\{M \sigma\} \in N/M$ is well-defined.
\end{proposition}

\begin{theorem}[see \cite{Lo}]
\label{1.7.2.thm}
For any class of Table \ref{t1} there is a one-to-one correspondence between the set of connected
components of the corresponding space $\R^\mu \setminus \Sigma$ and the set of
$M_\lambda$-conjugacy classes of involutions of
$\tilde H_{n-1}(V_\lambda, \Z)$ lying in  coset $M \sigma$. 
This correspondence associates with any component the class
of the action of complex conjugation in group $H_2(V_\lambda,\Z)$ for an arbitrary $\lambda$ from this component.
\end{theorem}

So, our results solve also the problem of the enumeration of conjugacy classes mentioned in the last theorem.

\section{Topological types of sets of lower values} 
\label{types}

In our pictures, we use a non-standard disposition of the coordinate frame, \,
\unitlength 1mm
\begin{picture}(7,5)
\put(0,0){\vector(1,0){5}}
\put(0,0){\vector(0,1){5}}
\put(5.2,0){$y$}
\put(1,4){$x$} \put(7,0){.}
\end{picture} 

By {\em ovals} we mean compact components of zero sets of functions $f_\lambda:\R^2 \to \R$. 

\subsection{$D_{2k}^-$}

The set $f^{-1}(0)$ and domains of constant signs of the polynomial $f$ look in this case as \   \ \unitlength 1mm
\begin{picture}(20,11)
\bezier{100}(0,0)(10,10)(20,0)
\bezier{100}(0,10)(10,0)(20,10)
\put(10,0){\line(0,1){10}}
\put(0,4){$+$}
\put(17,4){$-$}
\put(14,0){$+$}
\put(14,8){$+$}
\put(4,0){$-$}
\put(4,8){$-$}
\end{picture} \  \ (with curves replaced by lines if $k=2$).      
 The topological type of set $W(\lambda)$ is characterized in particular by the information which of the three asymptotic sectors of negative values of $f_\lambda$ are connected within this set. Depending on it, we have four cases listed in the following statement.
\begin{figure}
\unitlength 0.7mm
\begin{center}
{
\begin{picture}(80,34)
\bezier{250}(32,34)(22,12)(0,27)
\bezier{250}(32,0)(22,22)(0,7)
\bezier{250}(80,7)(60,17)(80,27)
\put(62.3,17){\circle{6}}
\put(46.5,17){\circle{6}}
\put(51,17){\dots}
\put(39,17){\circle{6}}
\put(74,16){$-$}
\put(60.4,16){$-$}
\put(44.5,16){$-$}
\put(37.3,16){$-$}
\put(12,5){$-$}
\put(12,27){$-$}
\put(4,16){$+$}
\put(0,0){\line(1,0){80}}
\put(0,0){\line(0,1){34}}
\put(80,34){\line(-1,0){80}}
\put(80,34){\line(0,-1){34}}
\end{picture} \qquad  
\begin{picture}(80,34)
\bezier{250}(48,34)(58,12)(80,27)
\bezier{250}(48,0)(58,22)(80,7)
\bezier{250}(0,7)(20,17)(0,27)
\put(17,17){\circle{6}}
\put(25,17){\circle{6}}
\put(29,17){\dots}
\put(39,17){\circle{6}}
\put(2,16){$+$}
\put(15.3,16){$+$}
\put(23,16){$+$}
\put(37.3,16){$+$}
\put(63,4){$+$}
\put(63,28){$+$}
\put(72,16){$-$}
\put(0,0){\line(1,0){80}}
\put(0,0){\line(0,1){34}}
\put(80,34){\line(-1,0){80}}
\put(80,34){\line(0,-1){34}}
\end{picture}
}
\end{center}
\caption{Cases I and IV for $D_{2k}^-$}
\label{d80}
\end{figure}
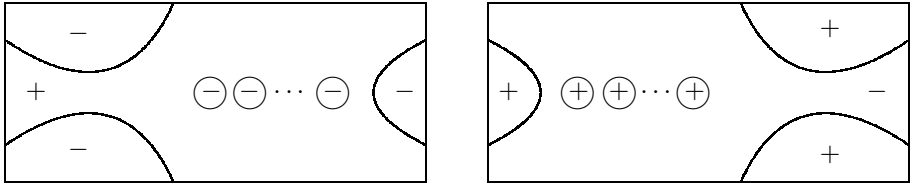

\begin{proposition}
\label{d2kmmm} All topological types of sets $W(\lambda)$, $\lambda \not \in \Sigma$,  in  case $D_{2k}^-$  belong to four families shown $($by the domains marked by \ $-$ \ $)$  in

I. $k$ pictures as in Fig.~\ref{d80} $($left$)$, the number of ovals varying from $0$ to $ k-1$.

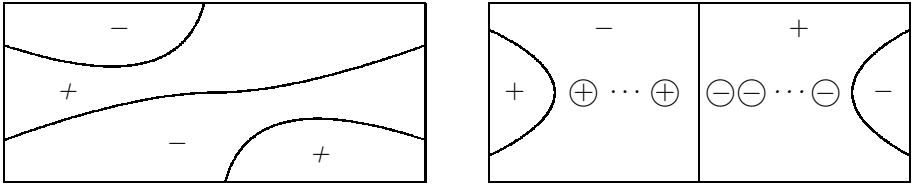
\begin{figure}
\unitlength 0.7mm
\begin{center}
{
\begin{picture}(80,34)
\bezier{200}(42,0)(47,19)(80,8)
\bezier{200}(0,26)(33,15)(38,34)
\bezier{200}(0,8)(25,17)(40,17)
\bezier{200}(80,26)(55,17)(40,17)
\put(20,28){$-$}
\put(10,16){+}
\put(31,6){$-$}
\put(58,4){+}
\put(0,0){\line(1,0){80}}
\put(0,0){\line(0,1){34}}
\put(80,34){\line(-1,0){80}}
\put(80,34){\line(0,-1){34}}
\end{picture} \qquad 
\begin{picture}(80,34)
\put(40,0){\line(0,1){34}}
\bezier{200}(80,5)(58,17)(80,29)
\bezier{200}(0,5)(25,17)(0,29)
\put(3,16){$+$}
\put(73,16){$-$}
\put(33.5,17){\circle{5}}
\put(31.5,15.7){$+$}
\put(18,17){\circle{5}}
\put(23,17){$\dots$}
\put(44,17){\circle{5}}
\put(50,17){\circle{5}}
\put(64,17){\circle{5}}
\put(16.1,15.7){$+$}
\put(42.2,15.5){$-$}
\put(48.0,15.5){$-$}
\put(54,17){\dots}
\put(62.0,15.5){$-$}
\put(20,28){$-$}
\put(57,28){$+$}
\put(0,0){\line(1,0){80}}
\put(0,0){\line(0,1){34}}
\put(80,34){\line(-1,0){80}}
\put(80,34){\line(0,-1){34}}
\end{picture} 
}
\end{center}
\caption{Cases II and III for $D_{2k}^-$}
\label{d81}
\end{figure}

II. Two pictures, one of which is shown in Fig.~\ref{d81} $($left$)$, and the other one is symmetric to it with respect to a horizontal mirror.

III. ${k \choose 2}$ pictures shown in Fig.~\ref{d81} $($right$)$, with the total number of ovals varying from $0$ to $k-2$.

IV. $k$ pictures shown in Fig.~\ref{d80} $($right$)$, the number of ovals varying from $0$ to $ k-1$.
\end{proposition}

\subsection{$D_{2k}^+$} The function of this type given in Table \ref{t1} vanishes only on  the line $\{y=0\}$ (vertical in our pictures).
\begin{figure}
\begin{center}
\begin{picture}(60,20)
\put(30,0){\line(0,1){20}}
\put(5,10){\circle{5}}
\put(12,10){\circle{5}}
\put(16,10){$\dots$}
\put(25,10){\circle{5}}
\put(35,10){\circle{5}}
\put(40,10){$\dots$}
\put(50,10){\circle{5}}
\put(10.5,9){$+$}
\put(3.5,9){$+$}
\put(23.5,9){$+$}
\put (12,2){$-$}
\put(33.5,9){$-$}
\put(48.5,9){$-$}
\put(45,2){$+$}
\put(0,0){\line(1,0){60}}
\put(0,0){\line(0,1){20}}
\put(60,20){\line(-1,0){60}}
\put(60,20){\line(0,-1){20}}
%\put(30,10){\oval(60,20)}
\end{picture}
\end{center}
\caption{$D_{2k}^+$}
\label{d84}
\end{figure}
\begin{proposition}
\label{d34}
All the topological types of sets $W(\lambda),$ where $\lambda$ are non-dis\-cri\-mi\-nant parameter values of a versal deformation of a singularity of type $D_{2k}^+$, are as shown in Fig.~\ref{d84}, where the numbers of ovals marked respectively by \ $+$ \ and \ $-$ \ can be any non-negative integers \ $p$ \ and \ $m$ \ such that \ $p + m \leq k-1$.
\end{proposition}

\subsection{$D_{2k-1}$}

We consider only  case $+D_{2k-1}$, case $-D_{2k-1}$ can be reduced to it by the multiplication of all functions $f_\lambda$  by $-1$. 
The zero set  of our function $f$ of type $+D_{2k-1}$ looks as \ \ 
\unitlength 1mm
\begin{picture}(18,9)
\put(10,0){\line(0,1){8}}
\bezier{70}(0,0)(4,4)(10,4)
\bezier{70}(0,8)(4,4)(10,4)
\put(-1,3){$+$}
\put(13,3){$+$}
\put(5,0){$-$}
\put(5,6){$-$}
\put(18,0){.}
\end{picture}
\unitlength 0.8mm
\begin{figure}
\begin{center}
{
\begin{picture}(60,30)
\bezier{200}(0,25)(27,10)(30,30)
\bezier{200}(0,5)(27,20)(30,0)
\put(35,15){\circle{5}}
\put(42,15){\circle{5}}
\put(46,15){$\dots$}
\put(55,15){\circle{5}}
\put(33.5,14){$-$}
\put(40.5,14){$-$}
\put(53.5,14){$-$}
\put(5,14){$+$}
\put(13,3){$-$}
\put(13,25){$-$}
\put(0,0){\line(1,0){60}}
\put(0,0){\line(0,1){30}}
\put(60,30){\line(-1,0){60}}
\put(60,30){\line(0,-1){30}}
\end{picture} \qquad
\begin{picture}(70,30)
\bezier{200}(0,5)(20,15)(0,25)
\bezier{200}(38,0)(33,20)(38,30)
\put(16,15){\circle{5}}
\put(21,15){$\dots$}
\put(30,15){\circle{5}}
\put(45,15){\circle{5}}
\put(50,15){$\dots$}
\put(60,15){\circle{5}}
\put(3,14){$+$}
\put(14.5,14){$+$}
\put(58.5,14){$-$}
\put(28.5,14){$+$}
\put(43.5,14){$-$}
\put(20,5){$-$}
\put(50,5){$+$}
\put(0,0){\line(1,0){70}}
\put(0,0){\line(0,1){30}}
\put(70,30){\line(-1,0){70}}
\put(70,30){\line(0,-1){30}}
\end{picture}
}
\end{center}
\caption{$+D_{2k-1}$}
\label{d9f}
\end{figure}
\begin{proposition}
\label{d45}
All the topological types of sets $W(\lambda)$, where $\lambda \in \R^{2k-1}$ are non-discriminant parameter values of  deformation $($\ref{dvd}$)$ of singularity $+D_{2k-1}$, are

I. $k$ types shown in Fig.~\ref{d9f} $($left$)$, where the number of ovals can be equal to $0,$ $1$, \dots, $k-1$;

II. ${k \choose 2}$ types shown in Fig.~\ref{d9f} $($right$)$,  where the numbers of ovals marked respectively  by \ $+$ and \ $-$  can be any non-negative integers \ $p$ and \ $m$ such that \ $p + m \leq k-2$.
\end{proposition}

\subsection{Singularities $E_\mu$}

Zero sets and distribution of signs of functions of classes   $+E_6, E_7$, and $E_8$ look respectively as \ \ 
\begin{picture}(11,7) 
\bezier{70}(0,0)(5,6)(10,0)
\put(5,3){\circle*{0.8}}
\put(3.5,0){$-$}
\put(3.5,5){$+$}
\end{picture} \ , \ \ 
\begin{picture}(17,8)
\put(0,4){\line(1,0){15}}
\bezier{70}(0,0)(5,4)(13,4)
\bezier{70}(0,8)(5,4)(13,4)
\put(11,6){{\small $+$}}
\put(11,0){$-$}
\put(0,1.5){{\small $+$}}
\put(0,4.3){$-$}
\put(11.5,4){\circle*{0.8}}
\end{picture} ,  and  \ \begin{picture}(12,7) 
\bezier{50}(5,3)(10,3)(10,0)
\bezier{50}(5,3)(0,3)(0,6)
\put(5,3){\circle*{0.8}}
\put(6.5,4){$+$}
\put(0.5,0){$-$}
\end{picture} .
Case $-E_6$ can be reduced to $+E_6$ by the multiplication of all functions (\ref{e6vd}) by $-1$ and simultaneous reflection $x \mapsto -x$. 

\begin{figure}
\unitlength 1mm
\begin{center}
{
\begin{picture}(20,30)
\bezier{100}(0,12)(10,18)(20,12)
\put(8.5,5){$-$}
\put(9,26){$+$}
\put(0,0){\line(1,0){20}}
\put(0,0){\line(0,1){30}}
\put(20,30){\line(-1,0){20}}
\put(20,30){\line(0,-1){30}}
\end{picture} \quad 
\begin{picture}(20,30)
\bezier{100}(0,12)(10,18)(20,12)
\put(10,19){\circle{5}}
\put(8.5,18){$-$}
\put(8.5,5){$-$}
\put(8.5,26){$+$}
\put(0,0){\line(1,0){20}}
\put(0,0){\line(0,1){30}}
\put(20,30){\line(-1,0){20}}
\put(20,30){\line(0,-1){30}}
\end{picture} \quad 
\begin{picture}(20,30)
\bezier{100}(0,12)(10,18)(20,12)
\put(3.5,18){$-$}
\put(14,19){\circle{5}}
\put(5,19){\circle{5}}
\put(12.5,18){$-$}
\put(8.5,5){$-$}
\put(9,26){$+$}
\put(0,0){\line(1,0){20}}
\put(0,0){\line(0,1){30}}
\put(20,30){\line(-1,0){20}}
\put(20,30){\line(0,-1){30}}
\end{picture} \quad 
\begin{picture}(20,30)
\bezier{100}(0,12)(10,18)(20,12)
\put(4,19){\circle{4}}
\put(10,19){\circle{4}}
\put(16,19){\circle{4}}
\put(2.5,18){$-$}
\put(8.5,18){$-$}
\put(14.5,18){$-$}
\put(8.5,5){$-$}
\put(9,26){$+$}
\put(0,0){\line(1,0){20}}
\put(0,0){\line(0,1){30}}
\put(20,30){\line(-1,0){20}}
\put(20,30){\line(0,-1){30}}
\end{picture} \quad 
\begin{picture}(20,30)
\bezier{100}(0,12)(10,18)(20,12)
\put(10,8){\circle{5}}
\put(8.5,7){$+$}
\put(8.5,1){$-$}
\put(8.5,26){$+$}
\put(0,0){\line(1,0){20}}
\put(0,0){\line(0,1){30}}
\put(20,30){\line(-1,0){20}}
\put(20,30){\line(0,-1){30}}
\end{picture}
}
\end{center}
\caption{$+E_6$}
\label{e61}
\end{figure}
\begin{proposition}
\label{e6pro}
There are exactly five topological types of sets $W(\lambda)$, where $\lambda$ are non-discriminant  parameter values of deformation $($\ref{e6vd}$)$ of a function of class $+E_6$; these types are shown in Fig.~\ref{e61}. 
\end{proposition}

\begin{figure}
\unitlength 1mm
\begin{center}
{
\begin{picture}(20,30)
\bezier{100}(0,15)(10,18)(0,25)
\bezier{200}(0,5)(10,15)(20,15)
\put(0.8,18){$-$}
\put(14,2){$-$}
\put(14,25){$+$}
\put(0,0){\line(1,0){20}}
\put(0,0){\line(0,1){30}}
\put(20,30){\line(-1,0){20}}
\put(20,30){\line(0,-1){30}}
\end{picture} \quad
\unitlength 1mm
\begin{picture}(20,30)
\put(0.8,18){$-$}
\put(14,2){$-$}
\put(14,25){$+$}
\bezier{100}(0,15)(10,18)(0,25)
\bezier{200}(0,5)(10,15)(20,15)
\put(11,19){\circle{5}}
\put(9.5,18){$-$}
\put(0,0){\line(1,0){20}}
\put(0,0){\line(0,1){30}}
\put(20,30){\line(-1,0){20}}
\put(20,30){\line(0,-1){30}}
\end{picture} \quad
\unitlength 1mm
\begin{picture}(20,30)
\put(0.8,18){$-$}
\put(14,2){$-$}
\put(14,25){$+$}
\put(10,17){\circle{5}}
\put(8.5,16){$-$}
\put(17,20){\circle{5}}
\put(15.5,19){$-$}
\bezier{100}(0,15)(10,18)(0,25)
\bezier{200}(0,5)(10,15)(20,15)
\put(0,0){\line(1,0){20}}
\put(0,0){\line(0,1){30}}
\put(20,30){\line(-1,0){20}}
\put(20,30){\line(0,-1){30}}
\end{picture} \quad
\unitlength 1mm
\begin{picture}(20,30)
\put(0.8,18){$-$}
\put(14,2){$-$}
\put(14,25){$+$}
\put(8,17){\circle{4}}
\put(6.5,16){$-$}
\put(18,20){\circle{4}}
\put(16.5,19){$-$}
\put(13,19){\circle{4}}
\put(11.5,18){$-$}
\bezier{100}(0,15)(9,18)(0,25)
\bezier{200}(0,5)(10,15)(20,15)
\put(0,0){\line(1,0){20}}
\put(0,0){\line(0,1){30}}
\put(20,30){\line(-1,0){20}}
\put(20,30){\line(0,-1){30}}
\end{picture} \quad
\unitlength 1mm
\begin{picture}(20,30)
\bezier{100}(0,15)(10,18)(0,25)
\bezier{200}(0,5)(10,15)(20,15)
\put(11,7){\circle{5}}
\put(9.5,6){${+}$}
\put(14,2){$-$}
\put(14,25){$+$}
\put(0.8,18){$-$}
\put(0,0){\line(1,0){20}}
\put(0,0){\line(0,1){30}}
\put(20,30){\line(-1,0){20}}
\put(20,30){\line(0,-1){30}}
\end{picture} 
}
\end{center}
\caption{$E_7$}
\label{e71}
\end{figure}
\begin{proposition}
\label{e7pro}
The set of all  topological types of sets $W(\lambda)$, where $\lambda$ are non-dis\-cri\-mi\-nant parameter values of deformation $($\ref{e7vd}$)$, consists of five types shown in Fig.~\ref{e71} and other five types obtained from these ones by reflection in a horizontal mirror and simultaneous change of all signs \ $+$ \ and \ $-$ \ .
\end{proposition}

\begin{figure}
\unitlength 1mm
\begin{center}
{
\begin{picture}(35,20)
\bezier{100}(0,15)(6,10)(15,10)
\bezier{100}(15,10)(20,10)(30,5)
\put(25,15){$+$}
\put(4,2){$-$}
\put(0,0){\line(1,0){30}}
\put(0,0){\line(0,1){20}}
\put(30,20){\line(-1,0){30}}
\put(30,20){\line(0,-1){20}}
\end{picture} \qquad
\begin{picture}(35,20)
\bezier{100}(0,15)(6,10)(15,10)
\bezier{100}(15,10)(20,10)(30,5)
\put(25,15){$+$}
\put(4,2){$-$}
\put(0,0){\line(1,0){30}}
\put(0,0){\line(0,1){20}}
\put(30,20){\line(-1,0){30}}
\put(30,20){\line(0,-1){20}}
\put(15,14){\circle{5}}
\put(13.5,13){$-$}
\end{picture} \qquad
\begin{picture}(35,20)
\bezier{100}(0,15)(6,10)(15,10)
\bezier{100}(15,10)(20,10)(30,5)
\put(25,15){$+$}
\put(4,2){$-$}
\put(0,0){\line(1,0){30}}
\put(0,0){\line(0,1){20}}
\put(30,20){\line(-1,0){30}}
\put(30,20){\line(0,-1){20}}
\put(15,14){\circle{5}}
\put(13.5,13){$-$}
\put(7,15){\circle{5}}
\put(5.5,14){$-$}
\end{picture} \quad
\begin{picture}(35,25)
\bezier{100}(0,15)(6,10)(15,10)
\bezier{100}(15,10)(20,10)(30,5)
\put(25,16){$+$}
\put(4,2){$-$}
\put(0,0){\line(1,0){30}}
\put(0,0){\line(0,1){20}}
\put(30,20){\line(-1,0){30}}
\put(30,20){\line(0,-1){20}}
\put(15,14){\circle{5}}
\put(13.5,13){$-$}
\put(7,15){\circle{5}}
\put(5.5,14){$-$}
\put(22,13){\circle{5}}
\put(20.5,12){$-$}
\end{picture} \qquad
\begin{picture}(35,25)
\bezier{100}(0,15)(6,10)(15,10)
\bezier{100}(15,10)(20,10)(30,5)
\put(25,16){$+$}
\put(4,2){$-$}
\put(0,0){\line(1,0){30}}
\put(0,0){\line(0,1){20}}
\put(30,20){\line(-1,0){30}}
\put(30,20){\line(0,-1){20}}
\put(13,14){\circle{5}}
\put(11.5,13){$-$}
\put(7,15){\circle{5}}
\put(5.5,14){$-$}
\put(20,13){\circle{5}}
\put(18.5,12){$-$}
\put(26.5,11){\circle{5}}
\put(25,10){$-$}
\end{picture} \qquad
\begin{picture}(35,25)
\bezier{100}(0,15)(6,10)(15,10)
\bezier{100}(15,10)(20,10)(30,5)
\put(3,16){$+$}
\put(25,2){$-$}
\put(0,0){\line(1,0){30}}
\put(0,0){\line(0,1){20}}
\put(30,20){\line(-1,0){30}}
\put(30,20){\line(0,-1){20}}
\put(20,15){\circle{5}}
\put(18.5,14){$-$}
\put(10,5){\circle{5}}
\put(8.5,4){$+$}
\end{picture} 
}
\end{center}
\caption{$E_8$}
\label{e81}
\end{figure}
\begin{proposition}
\label{e8pro}
The set of all topological types of sets $W(\lambda)$, where $\lambda$ are non-dis\-cri\-mi\-nant  parameter values of deformation $($\ref{e8vd}$)$, consists of six types shown in Fig.~\ref{e81} and other four types obtained by the central symmetry and simultaneous change of all signs $\pm$ from the four types shown in  Fig.~\ref{e81} which are not invariant under this operation.
\end{proposition}

\begin{remark} \rm
Some pictures of topological schemes of Figs.~\ref{e6pro}--\ref{e8pro} are geometrically non-realistic: e.g. by Bezout theorem no line can pass through three ovals.
\end{remark}

\section{Realization}
\label{realiz}

In this section we prove that all the topological types of sets $W(\lambda)$ listed in \S \ref{types} can be realized within the corresponding versal deformations. 
\begin{figure}
\unitlength 0.80mm
\linethickness{0.4pt}
\begin{center}
\begin{picture}(119.00,25.00)
\put(58,9.3){\circle*{1.33}}
\put(58,18.5){\circle*{1.33}}
\put(7.00,14.00){\circle*{1.33}}
\put(27.00,14.00){\circle*{1.33}}
\put(47.00,14.00){\circle*{1.33}}
\put(58,2){\line(0,1){23}}
\put(53,13){$+$}
\put(60,13){$-$}
\bezier{120}(27.00,14.00)(17.00,3.00)(7.00,14.00)
\bezier{120}(27.00,14.00)(17.00,25.00)(7.00,14.00)
\bezier{52}(27.00,14.00)(30.00,19.00)(37.00,19.00)
\bezier{44}(37.00,19.00)(43.00,19.00)(46.50,14.50)
\bezier{44}(27.00,14.00)(30.00,9.00)(37.00,9.00)
\bezier{40}(37.00,9.00)(43.00,9.00)(46.50,13.50)
\bezier{60}(47.50,14.50)(57,23)(66.50,14.50)
\bezier{60}(47.50,13.50)(57,5)(66.50,13.50)
\put(6.50,14.50){\line(-1,1){5}}
\put(6.50,13.50){\line(-1,-1){5}}
\put(1,12.5){$+$}
\put(17.00,14.00){\makebox(0,0)[cc]{$+$}}
\put(37.00,14.00){\makebox(0,0)[cc]{$+$}}
\put(67.00,14.00){\circle*{1.33}}
\put(87.00,14.00){\circle*{1.33}}
\put(107.00,14.00){\circle*{1.33}}
\bezier{100}(67.50,14.50)(77.00,24.00)(86.50,14.50)
\bezier{100}(67.50,13.50)(77.00,4.00)(86.50,13.50)
\bezier{100}(87.50,14.50)(97.00,24.00)(106.50,14.50)
\bezier{100}(87.50,13.50)(97.00,4.00)(106.50,13.50)
%\emline(107.50,14.50)(113.50,20.50)
\multiput(107.50,14.50)(0.12,0.12){50}{\line(0,1){0.12}}
%\end
%\emline(113.50,12.50)(107.50,18.50)
\multiput(113.50,7.50)(-0.12,0.12){50}{\line(0,1){0.12}}
%\end
\put(111.00,14.00){\makebox(0,0)[cc]{$-$}}
\put(97.00,14.00){\makebox(0,0)[cc]{$-$}}
\put(77.00,14.00){\makebox(0,0)[cc]{$-$}}
\put(26,4){$-$}
\put(26,20){$-$}
\put(85,4){$+$}
\put(85,20){$+$}
\end{picture}
\end{center}
\caption{Preliminary perturbation for $D_{2k}^-$, \  $k=7$}
\label{d89}
\end{figure}
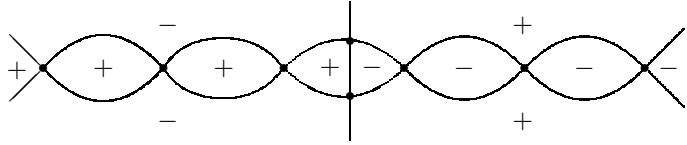

\subsection{$D_{2k}^-$} 
For any $j=0, \dots, k-1,$ the polynomial of type $D_{2k}^-$ indicated in Table \ref{t1} has a perturbation $\tilde f_j$, the set of zeros of which contains the vertical line $\{y=0\}$ and is topologically situated as
  shown in Fig.~\ref{d89}, where the number of intersection points of two non-vertical curves is equal to $k-1$, $j$ of these points lie to the left of the vertical line and remaining $k-1-j$ to the right of it. (In our picture $k=7$, $j=3$). Namely, in the versal deformation (\ref{dvd})   we need to set $\lambda_1=\lambda_2=0$ and choose remaining coefficients in such a way that the polynomial $\lambda_3+\lambda_4y + \dots + \lambda_{2k} y^{2k-3} - y^{2k-2}$ in one variable has $k-1$ different real double roots, exactly $j$ of which are negative.

Of course, this function $\tilde f_j$  belongs to the discriminant variety. Let us  somehow assign a sign $+$ or $-$ to any of its $k+1$ saddlepoints.  It follows easily from interpolation theorem and versality property (see e.g. Proposition 6.2 in \S~I.6 of \cite{APLT}) that for any such assignment we can perturb function $\tilde f_j$  arbitrarily weakly within our versal deformation in such a way that the  values of the new function at its critical points, to which all saddlepoints  of $\tilde f_j$ marked by $+$ (respectively, $-$) will move, become positive (respectively, negative). 

All desired non-discriminant perturbations will be obtained in this way from these $k$ functions $\tilde f_j$. Namely, to 
obtain the picture shown in Fig.~\ref{d80} (left)  with exactly $k-1$ ovals we place all intersection points of non-vertical curves
to the right of the vertical line (i.e. take $j=0$) and mark all $k+1$ intersection points of our curves by $+$ . To obtain the picture of the same Fig.~\ref{d80} (left) with $i < k-1$ ovals, we can replace this sign \ $+$ \ by \ $-$ \ at arbitrary $k-1-i$ \ intersection points of two non-vertical curves. The pictures of Fig.~\ref{d80} (right)  can be realized in a symmetric way, by placing all  intersection points of non-vertical components to the left of the vertical line and changing all the signs at the symmetric intersection points.
 
To obtain the picture of Fig.~\ref{d81} (left) we can choose $j$ arbitrarily, mark all intersection points to the right of the vertical line by $-$, all points to the left of it by $+$, the upper intersection points of this line by $+$ and the lower one by $-$ . Changing the markings of two last points we obtain the picture symmetric to it with respect to a horizontal line.  To obtain Fig.~\ref{d81} (right) with $p$ positive and $m$ negative ovals, $p + m \leq k-2$, we place exactly \ $p+1$ \ intersection points to the left of the  vertical line, mark by $+$ both saddlepoints on this line and exactly $m$ (arbitrary) intersection points of non-vertical curves to the right of it, and mark by $-$ all $p+1$ intersection points to the left of it and remaining \ $k-2-m-p$ \ points to the right of it.

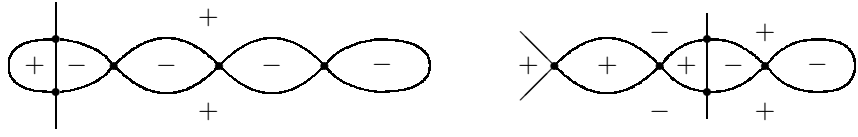
\begin{figure}
\unitlength 0.70mm
\linethickness{0.4pt}
\begin{center}
{
\begin{picture}(81.00,27.00)
\put(21.00,15.00){\circle*{1.33}}
\put(41.00,15.00){\circle*{1.33}}
\bezier{100}(40.50,14.50)(31.00,5.00)(21.50,14.50)
\bezier{100}(21.50,15.50)(31.00,25.00)(40.50,15.50)
\bezier{100}(60.50,14.50)(51.00,5.00)(41.50,14.50)
\bezier{100}(41.50,15.50)(51.00,25.00)(60.50,15.50)
\bezier{44}(61.50,15.50)(64.00,19.00)(71.00,20.00)
\bezier{44}(61.50,14.50)(64.00,11.00)(71.00,10.00)
\bezier{60}(71.00,10.00)(81.00,10.00)(81.00,15.00)
\bezier{60}(71.00,20.00)(81.00,20.00)(81.00,15.00)
\bezier{44}(20.50,14.50)(18.00,11.00)(11.00,10.00)
\bezier{44}(20.50,15.50)(18.00,19.00)(11.00,20.00)
\put(61.00,15.00){\circle*{1.33}}
\put(10.00,20.00){\circle*{1.33}}
\put(10.00,10.00){\circle*{1.33}}
\bezier{52}(9.33,10.00)(1.00,10.00)(1.00,15.00)
\bezier{52}(1.00,15.00)(1.00,20.00)(9.33,20.00)
\put(10.00,27.00){\line(0,-1){6.33}}
\put(10.00,19.33){\line(0,-1){8.66}}
\put(10.00,9.33){\line(0,-1){6.33}}
\put(6.00,15.00){\makebox(0,0)[cc]{$+$}}
\put(14.00,15.00){\makebox(0,0)[cc]{$-$}}
\put(31.00,15.00){\makebox(0,0)[cc]{$-$}}
\put(51.00,15.00){\makebox(0,0)[cc]{$-$}}
\put(70,14){$-$}
\put(37,5){$+$}
\put(37,23){$+$}
\end{picture} \qquad  
\begin{picture}(69.00,32.00)
\put(12.00,15.00){\circle*{1.33}}
\put(32.00,15.00){\circle*{1.33}}
\bezier{100}(31.50,15.50)(22.00,25.00)(12.50,15.50)
\bezier{100}(12.50,14.50)(22.00,5.00)(31.50,14.50)
\bezier{44}(32.50,14.50)(36.00,10.00)(42.00,10.00)
\bezier{40}(32.50,15.50)(36.00,20.00)(42.00,20.00)
\put(52.00,15.00){\circle*{1.33}}
\put(60,14){$-$}
\bezier{44}(52.50,14.50)(56,10)(62,10)
\bezier{40}(52.50,15.50)(56,20)(62,20)
\bezier{44}(51.50,14.50)(47,10)(42,10)
\bezier{40}(51.50,15.50)(47,20)(42,20)
\bezier{48}(62.00,20.00)(69.00,20.00)(69.00,15.00)
\bezier{48}(69.00,15.00)(69.00,10.00)(62.00,10.00)
\put(11.50,14.50){\line(-1,-1){6}}
\put(11.50,15.50){\line(-1,1){6}}
\put(41,5){\line(0,1){20}}
\put(41.00,10.00){\circle*{1.33}}
\put(41.00,20.00){\circle*{1.33}}
\put(7.00,15.00){\makebox(0,0)[cc]{$+$}}
\put(22.00,15.00){\makebox(0,0)[cc]{$+$}}
\put(37.00,15.00){\makebox(0,0)[cc]{$+$}}
\put(46.00,15.00){\makebox(0,0)[cc]{$-$}}
\put(30,5){$-$}
\put(30,20){$-$}
\put(50,5){$+$}
\put(50,20){$+$}
\end{picture}
}
\end{center}
\caption{Preliminary perturbations for $D_{2k}^+$ $(k=5)$ and $+D_{2k-1}$    $(k=5)$} 
\label{d8p9}
\end{figure}

\subsection{$D_{2k}^+$} For any $j=0, \dots, k-2,$ there is a perturbation $\tilde f_j$ of the function $f$ of this class, such that the set $\tilde f_j^{-1}(0)$ contains a vertical line and is topologically equivalent to the picture shown in Fig.~\ref{d8p9}  (left), with exactly $j$ self-intersection points of the non-vertical component to the left of the vertical line. (In our picture, we have $j=0$).
To realize Fig.~\ref{d84} with $p$ positive and $m$ negative ovals, $p + m \leq k-1$, we can take such a picture with $j=p$, mark by \ $+$ \ both intersection points of different components of its zero set and exactly $m-1$ self-intersection points to the right of the vertical line, and mark by \ $-$ \ all self-intersection points to the left of this line and also remaining \ $k-1-m-p$ \ self-intersection points to the right of it.

\subsection{$+D_{2k-1}$}
 For any $j=0, \dots, k-2,$ there is a perturbation $\tilde f_j$ of the polynomial of this class from Table \ref{t1}, zero set of which contains a vertical line and is topologically equivalent to the picture shown in Fig.~\ref{d8p9} (right), with $j$ self-intersection points of the non-vertical component to the left of the vertical line and \ $k-2-j$ \ such points to the right of it. (In our example $j=2$.) To realize Fig.~\ref{d9f} (left) with $p$ ovals, $p \in \{0, 1, \dots, k-1\}$, we can take $j=0$ and mark by \ $+$ \ both  intersection points of the vertical line with the non-vertical component, and exactly $p$ (arbitrary) self-intersection points of this component. To realize Fig.~\ref{d9f} (right) with $p$ positive and $m$ negative ovals, $p + m \leq k-2$, we can take the perturbation of Fig.~\ref{d8p9} with $j=p+1$, then  mark by \ $+$ \ both intersection points of the vertical line and exactly $m-1$ (arbitrary) self-intersection points of the other component lying to the right of the vertical line, and mark by \ $-$ \ all the remaining self-intersection points.
\begin{figure}
\begin{center}
{
\begin{picture}(50,26)
\bezier{300}(15,1)(40,31)(47,26)
\bezier{70}(47,26)(50,24)(47,22)
\bezier{250}(47,22)(25,14)(3,22)
\bezier{70}(3,22)(0,24)(3,26)
\bezier{300}(3,26)(10,31)(35,1)
\put(25.00,12.30){\circle*{1.33}}
\put(30.80,18.20){\circle*{1.33}}
\put(19.20,18.20){\circle*{1.33}}
\put(25.00,16.00){\makebox(0,0)[cc]{$-$}}
\put(40.00,22.00){\makebox(0,0)[cc]{$-$}}
\put(10.00,22.00){\makebox(0,0)[cc]{$-$}}
\put(24,2){$-$}
\end{picture}
\qquad \qquad \qquad
\begin{picture}(50,26)
\put(25.00,16.20){\circle*{1.33}}
\put(31.10,10.20){\circle*{1.33}}
\put(18.90,10.20){\circle*{1.33}}
\bezier{300}(10,1)(40,34)(47,26)
\bezier{70}(47,26)(50,24)(47,20)
\bezier{300}(47,20)(25,-1)(3,20)
\bezier{70}(3,20)(0,24)(3,26)
\bezier{300}(3,26)(10,34)(40,1)
\put(25.00,12.00){\makebox(0,0)[cc]{$+$}}
\put(39.00,20.00){\makebox(0,0)[cc]{$-$}}
\put(11.00,20.00){\makebox(0,0)[cc]{$-$}}
\put(24,2){$-$}
\end{picture}
}
\end{center}
\caption{Perturbations for $+E_6$}
\label{e6proof}
\end{figure}
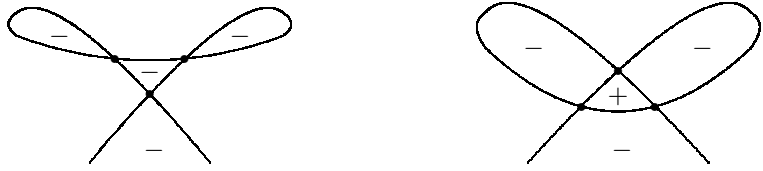
\subsection{$+ E_6$}
\label{e6reaa}
 By Examples 3 and 1 in \cite{GZ-1} (see also \cite{AC}), deformation (\ref{e6vd}) contains perturbations with zero sets topologically equivalent to the curves shown in Fig.~\ref{e6proof}. Acting as in the previous three subsections and marking by $+$ \ exactly 0, 1, 2 or 3 self-intersection points of this curve in Fig.~\ref{e6proof} (left) we obtain  topological types shown in the first four pictures of Fig.~\ref{e61}. Marking by $-$ all the self-intersection points in Fig.~\ref{e6proof} (right) we obtain the fifth picture of Fig.~\ref{e61}.

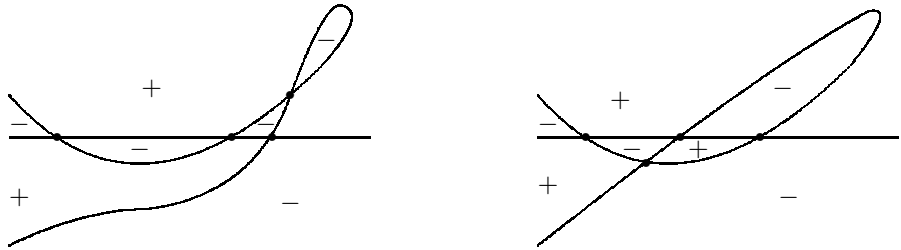
\begin{figure}
\begin{center}
{
\begin{picture}(60,40)
\put(0,18){\line(1,0){60}}
\bezier{120}(0,0)(10,5)(20,6)
\bezier{80}(20,6)(28,6)(35,10)
\bezier{250}(0,25)(22,0)(52,30)
\bezier{100}(52,30)(60,39)(55,40)
\bezier{100}(55,40)(52,40)(46,23)
\bezier{100}(46,23)(42,14)(35,10)
\put(0,7){$+$}
\put(20,15){$-$}
\put(41,19){$-$}
\put(51,33){$-$}
\put(22,25){$+$}
\put(0,19){$-$}
\put(45,6){$-$}
\put(8,18){\circle*{1.33}} 
\put(37,18){\circle*{1.33}} 
\put(43.7,18){\circle*{1.33}} 
\put(46.7,25){\circle*{1.33}}
\end{picture}
\qquad \qquad \qquad
\begin{picture}(60,40)
\bezier{250}(0,25)(22,0)(52,30)
\bezier{100}(52,30)(60,40)(55,39)
\bezier{350}(55,39)(38,30)(0,0)
\put(0,18){\line(1,0){60}}
\put(8,18){\circle*{1.33}} 
\put(23.7,18){\circle*{1.33}} 
\put(37,18){\circle*{1.33}} 
\put(18,13.7){\circle*{1.33}}
\put(39,25){$-$}
\put(14,15){$-$}
\put(25,15){$+$}
\put(0,9){$+$}
\put(0,19){$-$}
\put(12,23){$+$}
\put(40,7){$-$}
\end{picture}
}
\end{center}
\caption{Convenient perturbations for $E_7$}
\label{e7proof}
\end{figure}

\subsection{$E_7$} 
\label{e7reaa}
According to \cite{AC}, page 16, the singularity of this type has a perturbation $f_\lambda$ with  zero set as shown in Fig.~\ref{e7proof} (left). It follows directly from the normal forms of Table \ref{t1} and (\ref{e7vd}) that it has also perturbations as in Fig.~\ref{e7proof} (right). Marking by \ $+$ \ the leftmost intersection point  in Fig.~\ref{e7proof} (left) and additionally exactly  \ $0, 1, 2 $ or $3$ of other intersection points, we obtain the first  four pictures of Fig.~\ref{e71}. Marking the leftmost intersection point in Fig.~\ref{e7proof} (right) by \ $+$ \ and the other three intersection points by \ $-$ \ we obtain the fifth picture of Fig.~\ref{e71}. The involution of  deformation (\ref{e7vd}) sending any function $f_\lambda(x, y)$ to $-f_\lambda(-x, y)$ turns these five topological types  to five different ones.

\subsection{$E_8$} 
\label{e8reaa}
By \cite{AC}, page 17, there exist perturbations of singularities of type $E_8$, zero sets of which are topologically situated as in Fig.~\ref{e8proof} (left and right). Similarly to the previous subsection, we can obtain the topological types of the first five pictures of Fig.~\ref{e81} by certain markings of intersection points in Fig.~\ref{e8proof} (left).  Marking the rightmost intersection point in Fig.~\ref{e8proof} (right) by \ $+$ \ and other three ones by \ $-$ \ we obtain the sixth picture of Fig.~\ref{e81}. The involution \ $f_\lambda(x, y) \mapsto -f_\lambda(-x,-y)$ \ turns the middle four functions of this list to functions realizing the remaining topological types.

\unitlength 0.80mm
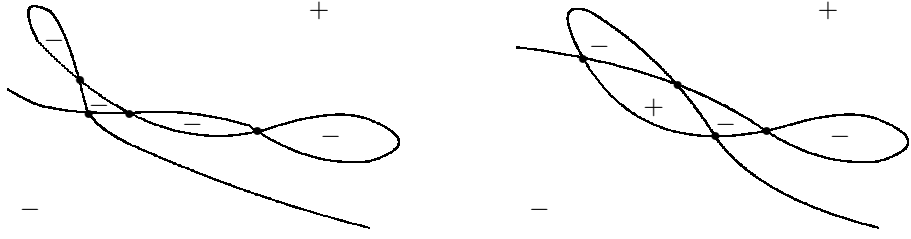
\begin{figure}
\begin{center}
{
\begin{picture}(70,36)
\bezier{70}(0,23)(5,20)(7,20)
\bezier{70}(7,20)(11,19)(20.2,19)
\bezier{100}(20.2,19)(30,20)(40,17)
%\bezier{250}(0,16)(25,22)(40,17)
\bezier{100}(40,17)(50,10)(60,11)
\bezier{70}(60,11)(70,14)(60,18)
\bezier{100}(60,18)(55,20)(45,17)
\bezier{150}(45,17)(25,10)(5,31)
\bezier{100}(5,31)(1,39)(7,36)
\bezier{100}(7,36)(10,33)(13,20)
\bezier{100}(13,20)(14,17)(20,14)
\bezier{150}(20,14)(40,5)(60,0)
\put(12,24.5){\circle*{1.5}}
\put(13.5,19){\circle*{1.5}}
\put(20.2,19){\circle*{1.5}}
\put(41.5,16){\circle*{1.5}}
\put(2,2){$-$}
\put(6,30){$-$}
\put(13.5,19.2){$-$}
\put(29,16){$-$}
\put(52,14){$-$}
\put(50,35){$+$}
\end{picture} \qquad \quad 
\begin{picture}(70,36)
\bezier{200}(0,30)(25,27)(40,17)
\bezier{100}(40,17)(50,10)(60,11)
\bezier{70}(60,11)(70,14)(60,18)
\bezier{100}(60,18)(55,20)(45,17)
\bezier{150}(45,17)(20,10)(10,30)
\bezier{100}(10,30)(6,38)(12,36)
\bezier{200}(12,36)(25,28)(33,15)
\bezier{150}(33,15)(40,5)(60,0)
\put(11,28){\circle*{1.5}}
\put(26.7,23.7){\circle*{1.5}}
\put(33,15.3){\circle*{1.5}}
\put(41.5,16){\circle*{1.5}}
\put(2,2){$-$}
\put(12,29){$-$}
\put(21,19){$+$}
\put(33,16){$-$}
\put(52,14){$-$}
\put(50,35){$+$}
\end{picture}
}
\end{center}
\caption{Perturbations for $E_8$}
\label{e8proof}
\end{figure}

\section{Homology calculations for $D_\mu$}
\label{esti}
\subsection{Simplicial resolution}

We study the cohomology groups of spaces $\R^\mu \setminus \Sigma$ by the method of \cite{V88}, see also \cite{Vbook}. First, following \cite{A70}, we use Alexander duality 
\begin{equation}
\label{alex}
\tilde H^j(\R^\mu\setminus \Sigma) \simeq \bar H_{\mu-1-j} (\Sigma)
\end{equation}
reducing them to the homology of discriminant sets; here $\tilde H^*$ means the cohomology group reduced modulo a point, and $\bar H_*$ the Borel--Moore homology. 

\begin{theorem}
\label{homdm}
Let $\Sigma$ be the discriminant variety of deformation $($\ref{dvd}$)$, then

1$)$ groups $\bar H_i(\Sigma, \Z_2)$ are trivial if $i \neq \mu-1$,

2$)$  $\bar H_{\mu-1}(\Sigma, \Z_2)$   is isomorphic to  $(\Z_2)^{\frac{(k+2)(k+1)}{2}}$ in case $D_{2k}^-$ and to 
$(\Z_2)^{\frac{(k+1)k}{2}-1}$ in cases $D_{2k}^+$ $(k \geq 2)$ and $\pm D_{2k-1}$ $(k \geq 3)$.
\end{theorem}

This theorem implies the completeness of lists of components of type $D_\mu$ given in Propositions \ref{d2kmmm}, \ref{d34} and \ref{d45}, and completes the proofs of these propositions and of parts of Theorems \ref{mthm} and \ref{mthm2} concerning singularities of types $D_\mu$.
Its statement 1) follows also from Theorem \ref{lotriv}.  \medskip

To prove Theorem \ref{homdm}, we construct a simplicial resolution of $\Sigma$. Let $T$ be the maximal number (over all $\lambda \in \Sigma$)  of different points $X_i \in \R^2$ such that $f_\lambda(X_i)=0$ and $df_\lambda(X_i)=0$. (It is easy to see that $T$ is equal to $k+1$ for $D_{2k}^-$ and to $k$ for $D_{2k}^+$  and $D_{2k-1}$.)
Consider a generic polynomial embedding $I: \R^2 \to \R^M$ into a space of very large dimension. The genericity of the embedding assumes that the convex hull of any $l \leq T$ different points $I(X_i)$, $X_i \in \R^2$, is a simplex of dimension $l-1$, and these simplices have only expected intersections (i.e. the intersection of any two simplices is their common face spanned by their common vertices).

Define  the $T$-th self-join  $(\R^2)^{*T} $ of $\R^2$ as the union of all these simplices in $\R^M$. 
It is easy to see that spaces $(\R^2)^{*T}$  defined in this way by different generic embeddings $I$ are canonically homeomorphic to one another.
The simplicial resolution $\sigma$ of the variety $\Sigma$ is a subset of the Cartesian product $(\R^2)^{*T} \times \R^\mu$. Namely, for any point $\lambda \in \Sigma$ we take all points $X_i \in \R^2$ such that $f_\lambda(X_i) =0$ and $d f_\lambda(X_i) =0$, then define the simplex $\Delta(\lambda) \subset (\R^2)^{*T}$ as the convex hull of all corresponding points $I(X_i)$, and consider the simplex $\Delta(\lambda) \times \lambda \subset (\R^2)^{*T} \times \R^\mu$.  Space $\sigma$ of our simplicial resolution is defined as the union of all these simplices $\Delta(\lambda) \times \lambda$ over all $\lambda \in \Sigma$.

\begin{proposition}[see e.g. \cite{Vbook}] \label{mprop}
The map $\sigma \to \Sigma $ defined by the obvious projection $(\R^2)^{*T} \times \R^\mu \to \R^\mu$ is proper and induces an isomorphism of Borel--Moore homology groups, $\bar H_*(\sigma) \stackrel{\sim}{\longrightarrow} \bar H_*(\Sigma)$. \hfill $\Box$
\end{proposition}

This space $\sigma$ admits an increasing filtration: its subspace $\sigma_p \subset \sigma$ is defined as  the union of all simplices \ $\Delta \times \lambda \subset (\R^2)^{*T} \times \R^\mu$  \ spanned by no more than $p$ points $I(X_i) \times \lambda$ such that $f_\lambda(X_i)=0$, $df_\lambda(X_i)=0$. (Notice that such a simplex $\Delta$ is not necessarily a simplex $\Delta(\lambda)$ from the previous construction: it can be also a face of such a simplex.)
Let us calculate the spectral sequence $E_{p, q}^r$ converging to group $\bar H_*(\sigma,\Z_2) \simeq \bar H_*(\Sigma,\Z_2)$ and generated by our filtration $\{\sigma_p\}$. By definition its term $E_{p, q}^1$ is isomorphic to $\bar H_{p+ q}(\sigma_p \setminus \sigma_{p-1},\Z_2)$.

For any $p=1, \dots, T$ consider configuration space  $B(\R^2,p)$  of unordered subsets of cardinality $p$ in $\R^2$. Denote by $\Xi_p$ the subset in $B(\R^2,p) \times \R^\mu$ consisting of all pairs of the form (a $p$-configuration in $\R^2$, a point $\lambda \in \Sigma$) such that $f_\lambda (X)=0$ and $df_\lambda(X) =0$ for all points $X$ of this configuration. By the Vieta mapping  space $B(\R^2,p)$ can be identified with an open subset in $\R^{2p}$, therefore $\Xi_p$ can be considered as a semialgebraic subset in $\R^{2p+\mu}$.

\begin{proposition}[cf. \cite{Vbook}]
\label{mprop2}
For any $p=1, \dots, T,$ \  difference $\sigma_p \setminus \sigma_{p-1}$ is the space of a fiber bundle with base $\Xi_p$ and fibers equal to  open $(p-1)$-dimensional simplices; in particular $\bar H_N(\sigma_p \setminus \sigma_{p-1}, \Z_2) \simeq \bar H_{N-p+1}(\Xi_p, \Z_2)$ for any $N$. \hfill $\Box$
\end{proposition}

\subsection{Groups $\bar H_*(\Xi_p,\Z_2)$ in the case $D_{2k}^-$}
Parameter  space $\R^{2k}$ of deformation  (\ref{dvd}) of class $D^-_{2k}$ consists of the hyperplane $\R^{2k-1}_0$ distinguished by the condition $ \lambda_2=0$ and two half-spaces $\R^{2k}_+$ and $\R^{2k}_-$ in which $\lambda_2>0$  (respectively, $\lambda_2<0$). 

\begin{proposition}
\label{hyperb}
For any $p=1, \dots, k,$ each of two spaces $\Xi_p \cap \left( B(\R^2,p) \times \R^{2k}_+\right)$ and $\Xi_p \cap \left(B(\R^2,p) \times \R^{2k}_-\right)$ consists of \ $p+1$ \ connected components homeomorphic to open $(2k-p)$-dimensional balls. For \ $p>k$ \ these spaces are empty.
\end{proposition}

\begin{proof} Consider only spaces $\Xi_p \cap \left( B(\R^2,p) \times \R^{2k}_-\right)$, as  spaces $\Xi_p \cap \left(B(\R^2,p) \times \R^{2k}_+\right)$ are symmetric to them. The map sending any point  \ (a $p$-configuration $ \rho \in B(\R^2,p); \lambda \in \Sigma)$ \ of such a space to coordinate $\lambda_2$ of $\lambda$ makes this space a  fiber bundle over the interval $(-\infty,0)$.
Indeed, the multiplicative group of positive numbers acts on it by quasi-homogeneous dilations $$t: (x, y) \to (t^{k-1}x, ty), \ 
t: f_\lambda(x, y) \to t^{-(2k-1)}f_\lambda(t^{k-1}x, ty), $$ and identifies all fibers of this map with one another.
 Therefore it is enough to consider only one its fiber consisting of such pairs with $\lambda_2=-1$. Let $(\rho,\lambda)$ be a point of this fiber. By the condition $\partial f_\lambda/\partial x =0$ all points $X_i$ of the $p$-configuration $\rho$ lie in the hyperbola $\{2xy=1\}$. If $p < k$ then by the interpolation theorem for any collection of $p$ different points in this hyperbola we can choose the remaining coefficients $\lambda_1, \lambda_3, \dots, \lambda_{2k}$ of polynomial (\ref{dvd}) in such a way that it will satisfy the conditions $f_\lambda =0, \partial f_\lambda/\partial y =0$ at all these \ $p$ \ points. Moreover, the set of points $\lambda$ satisfying these conditions is an affine subspace of codimension $2p$ in the space $\R^{2k-1}$ of these remaining coefficients.
Therefore space $\Xi_p \cap \left( B(\R^2,p) \times \R^{2k}_-\right)$   is homeomorphic to the product of spaces $(-\infty,0),$ $\R^{2k-1-2p}$, and the configuration space of $p$-point subsets of a hyperbola.

Now let \ $p=k$. Consider function (\ref{dvd}) as a polynomial of degree $2$ in $x$; its discriminant  is equal to
\begin{equation} \label{diss}
-\frac{1}{y^2}\left(\frac{1}{4} -\lambda_1y - \lambda_3y^2 - \lambda_4y^3 - \dots - \lambda_{2k} y^{2k-1} + y^{2k}\right) . 
\end{equation} If  $(x, y)$ is a critical point  of $f_\lambda$ with critical value $0$  then $y$ is a double root  of function  (\ref{diss}). Therefore if $f_\lambda$ has $k$ discriminant critical points then the polynomial in brackets should be the square of a monic polynomial of degree $k$, all roots of which are real and different. By Vieta formula the product of these roots should be equal to $1/2$ or $-1/2$. It is easy to see that conversely, if the polynomial in (\ref{diss}) is the square of such a polynomial, then $f_\lambda$ has $k$ discriminant points. The space of all polynomials satisfying these conditions consists again of $p+1 \equiv k+1$ components homeomorphic to open balls (of dimension $k-1$).

Function (\ref{diss}) cannot have more than $k$ multiple zeros, therefore functions $f_\lambda$ with $\lambda_2 < 0$ cannot have more than $k$ critical points with critical value 0. \end{proof}

\begin{definition} \rm
For any integers $p$ and $j$ such that $1\leq p \leq k$, $0 \leq j \leq p$, denote by $\Delta_{p, j}^+$ (respectively, $\Delta_{p, j}^-$) the component of the set $\Xi_p \cap \left(B(\R^2,p) \times \R^{2k}_+\right)$ (respectively, $\Xi_p \cap \left(B(\R^2,p) \times \R^{2k}_-\right)$) consisting of points $(\rho,\lambda)$ of this set such that the configuration $\rho \in B(\R^2,p)$ has exactly $j$ points in the left-hand part $\{y<0\}$ of $\R^2$.
\end{definition}

Consider now spaces $\Xi_p \cap \left( B(\R^2,p) \times \R^{2k-1}_0\right)$
of points $(\rho,\lambda) \in \Xi_p$ such that coordinate $\lambda_2$ of $\lambda$ is equal to $0$. The condition $\partial f_\lambda/\partial x =0$ for such $\lambda$ implies that all points of configuration $\rho$ lie in the cross $\{xy=0\}$. The set of critical points of $f_\lambda$ with critical value 0 is symmetric with respect to the axis $\{x=0\}$. The set of such points lying in the axis $\{y=0\}$ can consist either of one point (then it is the origin, and $\lambda_1=\lambda_3=0$) or of exactly two symmetric points (this happens if and only if $\lambda_1=0$ and $ \lambda_3<0$). In the latter case $x$-coordinates of these critical points are equal to $\pm \sqrt{-\lambda_3}$.

\begin{proposition}
\label{prd2kmb}
For any $p=1, \dots, k-1,$ the space $\Xi_p \cap \left(B(\R^2,p) \times \R^{2k-1}_0\right)$ is the union of the following strata:

a$)$ $p+1$ spaces \ $\Gamma_{p, j}$, \ $j=0, \dots, p,$ diffeomorphic to $\R^{2k-p-1}$ and consisting of pairs \ $(\rho, \lambda) \in \Xi_p$ such that all $p$ points of the configuration $\rho$ lie in the line $\{x=0\}$, $j$ of them in the negative part $\{y<0\}$ of this line and other \ $p-j$ \ in its positive part;

b$)$ $p$ spaces $\Lambda_{p, j},$ $j=0, \dots, p-1$,  diffeomorphic to $\R^{2k-p-2}$ and consisting of pairs $(\rho, \lambda) \in \Xi_p$ such that all points of the configuration $\rho$ lie in the line $\{x=0\}$, $j$ of them in the negative part, $p-j-1$ in the positive part, and one is equal to $0$;

c$)$ $2p$ spaces $\Theta_{p, j}^{\pm}$, $j=0, \dots, p-1$, diffeomorphic to $\R^{2k-p-1}$ and consisting of pairs $(\rho,\lambda) \in \Xi_p$ such that \ $p-1$ \ points of configuration $\rho$ are non-zero points of the line $
\{x=0\}$, exactly $j$ of them lie in the negative part of this line, while the $p$-th point of $\rho$ is a non-zero point of the line $\{y=0\}$, and specifically of its half indicated by the upper index \ $+$ \ or \ $-$ \ of the notation  $\Theta_{p, j}^{\pm}$;

d$)$ $p-1$ spaces $\Omega_{p, j}$, $j=0, \dots, p-2$, diffeomorphic to $\R^{2k-p}$ and consisting of pairs \ $(\rho,\lambda) \in \Xi_p$ such that  the configuration $\rho$ contains two opposite points of the line $\{y=0\}$ and \ $p-2$ \ non-zero points of the line $\{y=0\}$, exactly $j$ of them in the negative part of this line. 

If $p=k$ then $\Xi_p \cap \left(B(\R^2,p) \times \R^{2k-1}_0\right)$ consists of only $2k$ spaces of type c$)$ and $k-1$ spaces of type d$)$; if $p=k+1$ then it consists of only $k$ spaces of type d$)$.
\end{proposition}

\noindent
{\it Proof} is elementary. \hfill $\Box$

\begin{proposition}
The closures of sets $\Delta_{p, j}^{\pm}$ in $B(\R^2,p) \times \R^{2k}$ do not contain any points of  sets $\Omega_{p,i}$.
\end{proposition}

\noindent
\begin{proof} Let $f_\lambda$ be a polynomial of the form (\ref{dvd}) participating in a stratum $\Omega_{p,i}$, i.e. having two critical points with critical value 0 on the line $\{y=0\}$. These two points are then Morse saddlepoints, and there are no other critical points of $f_\lambda$ in a neighbourhood of this line. Set $f_\lambda^{-1}(0)$ consists in this neighbourhood of this vertical line and two other smooth components passing through these saddlepoints and transversal to the vertical line. Suppose that $f_{\tilde \lambda}$ is another function of the form (\ref{dvd}) with $\tilde \lambda$ very close to $\lambda$, also having two critical points with value $0$, which are close to these saddlepoints of $f_\lambda$. Structural stability of Morse functions implies easily that the topological picture of the zero level set of $f_{\tilde \lambda}$  is the same as that of $f_\lambda$ in our neighbourhood of the line $\{y=0\}$. In particular, the total intersection number of set $f^{-1}_{\tilde \lambda}(0)$ and the vertical line passing through one of these two critical points of $f_{\tilde \lambda}$ is at least 3. Since all functions (\ref{dvd}) are quadratic on $x$, this implies that  this vertical line should belong to this zero set; but then $f_{\tilde \lambda}$ cannot be a function participating in a  stratum of the class $\Delta^{\pm}_{p, j}$. \end{proof}

So, the space $\Xi_p$ consists of $p-1$ isolated cells $\Omega_{p, j}$ and the union of remaining spaces $\Delta_{p, j}^{\pm}$, $\Gamma_{p, j}$, $\Lambda_{p, j}$, and $\Theta_{p, j}^{\pm}$. Denote the latter union by $\tilde \Xi_p$. To calculate its Borel--Moore homology group, let us filter it by the unions of cells of different dimensions. The lowest term $F_0$ of this filtration consists of cells $\Lambda_{p, j}$, the next term $F_1$ contains additionally cells $\Gamma_{p, j}$ and $\Theta_{p, j}^{\pm}$, and $F_2$ contains additionally cells $\Delta_{p, j}^{\pm}$. If $p < k$ then term ${\mathcal E}^1$ of the spectral sequence defined by this filtration and calculating the group $\bar H_*(\tilde \Xi_p, \Z_2)$ consists of three non-trivial groups ${\mathcal E}^1_{a,2k-p-2}$, $a=0,1,2,$ combined by the differential $\partial^1$ of the spectral sequence into a complex of the form 
\begin{equation}
\label{rss}
(\Z_2)^p \leftarrow (\Z_2)^{3p+1} \leftarrow (\Z_2)^{2p+2} \ .
\end{equation}
If $p=k$ then $F_0$ is empty, and the similar complex looks as \ $0 \leftarrow (\Z_2)^{2k} \leftarrow (\Z_2)^{2k+2}$. All these groups are generated by classes of our cells. 

\begin{proposition}
\label{cor99}
The differentials of complex $($\ref{rss}$)$ for any $p \leq k-1$ are as follows:
$$\begin{array}{rclcrcl}
\partial \Delta_{p, j}^- & = & \Gamma_{p, j} + \Theta_{p, j}^+  + \Theta_{p,j-1}^- & \qquad &
\partial \Delta_{p, j}^+ & =&  \Gamma_{p, j} + \Theta_{p,j-1}^+  + \Theta_{p, j}^- \\
\partial \Theta_{p, j}^- & = &\Lambda_{p, j} 
 & \qquad &
\partial \Theta_{p, j}^+ & = & \Lambda_{p, j}  \\
\partial \Gamma_{p, j} & = & \Lambda_{p, j} + \Lambda_{p,j-1} 
& \qquad &
\partial \Lambda_{p, j} & = & 0 \ .
\end{array}
$$
For $p=k$ the differentials of the corresponding two-term complex are as follows:

\noindent
$
\partial \Delta_{k, j}^-  = \Theta_{k, j}^+  + \Theta_{k, j-1}^- \ , \quad 
\partial \Delta_{k, j}^+  =  \Theta_{k, j-1}^+  + \Theta_{k, j}^- \ .
$

In all these formulas we mean that $\Theta^{\pm}_{p,-1} = 0$, $\Theta^{\pm}_{p,p} =0$, $\Lambda_{p,-1}=0$, and $\Lambda_{p,p}=0$.
\end{proposition}

\noindent
{\it Proof.} All these boundary coefficients follow elementary from the definition of our cells. \hfill $\Box$

\begin{corollary}
For any $p =1, \dots, k-1$,  group $\bar H_*(\tilde \Xi_p, \Z_2)$ is isomorphic to $\Z_2$ in dimension $2k-p$ and is trivial in all other dimensions; this group is generated by the class of the union of all strata $\Delta_{p, j}^{\pm}$ over all indices $j=0, \dots, p$ and  signs $\pm$. The group $\bar H_*(\tilde \Xi_k,\Z_2)$ is isomorphic to $(\Z_2)^2$ in dimension $k$ and is trivial in all other dimensions; its two generators are defined by unions over \ $j$ \ of spaces $\Delta_{k,j}^+$ and $\Delta_{k,j}^-$ separately.
\end{corollary}

Considering additionally $(2k-p)$-dimensional cells $\Omega_{p, j}$ and applying Proposition \ref{mprop2}, we obtain the following result.

\begin{corollary}
Term $E_{p, q}^1$ of our spectral sequence converging to  group $\bar H_*(\sigma,\Z_2)$ and  defined by  filtration $\{\sigma_p\}$  is trivial if $p+ q \neq 2k-1$ or $p>k+1$. Terms $E^1_{p, 2k-1-p}$ are equal to
$(\Z_2)^p$ if $p \leq k-1$, to $(\Z_2)^{k+1}$ if $p=k$, and to \ $(\Z_2)^k$ if $p=k+1$. \hfill $\Box$
\end{corollary}

Theorem \ref{homdm} in case $D_{2k}^-$ follows immediately from this corollary.

\subsection{Cases $D_{2k}^+$ and $D_{2k-1}$}

\begin{proposition}
\label{prd2kpa}
Exact analogues of all statements of Propositions \ref{hyperb}--\ref{cor99} remain valid for spaces $\Xi_p$ arising in the simplicial resolution of discriminant sets of singularities $D_{2k}^+$ and $D_{2k-1}$, with only two exceptions: 

1$)$ the analogues of $2(k+1)$ spaces $\Delta_{k,j}^{\pm}$, $2k$ spaces $\Theta^{\pm}_{k,j}$ and $k$ spaces $\Omega_{k+1,j}$ vanish;

2$)$ dimensions of all remaining cells in case $D_{2k-1}$ are by 1 lower than dimensions of their namesakes from the study of $D_{2k}$ singularities.
\end{proposition}

Indeed, all reductions for $D_{2k}^+$ mentioned in statement 1) follow from the fact that the polynomial as in the brackets of formula (\ref{diss}) but with a different sign of the leading term cannot be a multiple of  the square of a real polynomial of degree $k$. A fortiori this is impossible in case $D_{2k-1}$ when the corresponding polynomial in brackets has a smaller degree. \hfill $\Box$

\begin{corollary}
\label{cor98}
For any \ $p=1, \dots, k-1$ the mod 2 Borel--Moore homology group of the union $\tilde \Xi_p$ of all sets  $\Delta_{p, j}^{\pm}$, $\Gamma_{p, j}$, $\Lambda_{p, j}$ and $\Theta_{p, j}^{\pm}$ arising in our simplicial resolution of the discriminant variety of a singularity of type $D_{2k}^+$ $($respectively,  $D_{2k-1})$ is isomorphic to $\Z_2$ in dimension $2k-p$ $($respectively, $2k-p-1)$ and is trivial in all other dimensions. For $p\geq k$ this group is trivial in all dimensions.
\end{corollary}

Considering additionally cells $\Omega_{p, j}$, we obtain the following result.

\begin{corollary}
Term $E_{p, q}^1$ of our spectral sequence converging to group $\bar H_*(\sigma,\Z_2)$, where $\sigma$ is the simplicial resolution of the discriminant variety of type $D_{2k}^+$ $($respectively,  $D_{2k-1})$, is trivial if $p+ q \neq 2k-1$ $($respectively, $p+ q \neq 2k-2)$ or $p>k$. Terms $E^1_{p, 2k-1-p}$ $($respectively, $E^1_{p,2k-2-p})$ are equal to \ $(\Z_2)^p$ \ if \ $p \leq k-1$\  and to \ $(\Z_2)^{k-1}$ \ if \ $p=k$.
\end{corollary}

Theorem \ref{homdm} in cases $D_{2k}^+$ \ and \ $D_{2k-1}$ follows immediately from this corollary.

\section{Case of $E_\mu$ singularities}
\label{estie}

In this section we prove that the correspondence between connected components of spaces $\R^\mu \setminus \Sigma$  of types $E_\mu$ and topological types of sets $W(\lambda)$ presented in \S~\ref{types} is one-to-one. The proof uses a combinatorial computer program (see \cite{pro2}) counting all collections of certain topological characteristics of non-discriminant morsifications of  function singularities. Let us describe this program; for a more detailed description see \cite{APLT}, \S V.8 and \cite{fuchs}, \S 3.4.

\subsection{Virtual morsifications and a program handling them}
\label{adad}

Let $f$ be one of polynomials from Table \ref{t1} with Milnor number $\mu$, \ $\R^\mu$ \ be the parameter space of its versal deformation from the list (\ref{avd})--(\ref{e8vd}), and $f_\lambda$, $\lambda \in \R^\mu$, be a  perturbation of $f$, all critical points of which in $\C^2$ are Morse and have $\mu$ different critical values not equal to 0. The group $H_2( V_\lambda,\Z) $ (where $V_\lambda$ is the corresponding Milnor fiber, see \S \ref{woL}) is then isomorphic to $\Z^\mu$ and is generated by {\em vanishing cycles} defined by a system of non-intersecting paths connecting the non-critical value $0$ with all these critical values, see e.g. \cite{AVGZ2}, \cite{APLT}. 

\begin{definition} \rm
A system of paths is {\em standard} if 

1) any two paths connecting 0 with complex conjugate critical values are complex conjugate to one another,

2) all paths connecting 0 with real critical values lie (except for their endpoints) in the domain of $\C^1$ where \ $\mbox{Im } z$ \ is positive but lower than absolute values of  imaginary parts of all non-real critical values of $f_\lambda$.
\end{definition}

All $\mu$ vanishing cycles defined by a standard system of paths can be {\it canonically oriented} by the following conditions:

1) canonical orientations of cycles vanishing in real critical points are induced by the fixed orientation of $\R^3$, see \S V.1.6   of \cite{APLT};

2) homology classes of any two vanishing cycles defined by mutually complex conjugate paths are mapped to one another by the complex conjugation in $\C^3$;

3) the intersection matrix of (naturally ordered) vanishing cycles is lexicographically maximal in the class of such matrices defined by systems of orientations satisfying previous two conditions;

4) if $f_\lambda$ has no real critical points, then the string of intersection indices of vanishing cycles with the naturally oriented submanifold $V_\lambda \cap \R^3$ is lexicographically maximal in the class of such strings defined by systems of orientations satisfying previous three conditions.

\begin{definition} \rm
A {\it virtual morsification} related to $f_\lambda$ is a collection of topological characteristics consisting of 
\begin{itemize}
\item the matrix of intersection indices of canonically oriented vanishing cycles in $V_\lambda$ defined by a standard system of paths, 
\item Morse indices of all real critical points of $f_\lambda$ ordered by increase of their critical values, and 
\item the number of  these points with negative critical values. 
\end{itemize}
\end{definition}

\begin{remark} \rm
If $f_\lambda$ has none or one pair of non-real critical points then there is only one virtual morsification related to $f_\lambda$. 
\end{remark}

The {\em elementary surgeries} of virtual morsifications include three moves of these data reflecting the topological surgeries of corresponding real morsifications, namely
\begin{itemize}
\item[(s1)] collisions of neighboring real critical values (after which the corresponding critical points either meet and  go into the complex domain or change the order in $\R^1$ of their critical values), 
\item[(s2)] collisions of complex conjugate critical values at  line $\R^1$ (after which the corresponding critical points either meet and come to $\R^2$ or miss one another while the imaginary parts of the critical values change their signs), 
\item[(s3)] jumps of real critical values over 0; \\ \hspace*{5cm} and additionally
\item[(s4)] specifically virtual surgeries within the classes of virtual morsifications related to one and the same real morsification, caused by changes of paths going from 0 to imaginary critical values.
\end{itemize}

 Any virtual morsification determines the set of surgeries (s1)--(s4) which can be applied to it and to the related real morsification, and also the results of all these surgeries.

All such possible  changes of virtual morsifications of {\em simple} singularities can be incarnated by actual surgeries of  functions $f_\lambda$: this follows from the properness of the {\em Looijenga map}  $\C^\mu \to \mbox{Sym}^\mu(\C^1) \simeq \C^\mu$ which sends any parameter value $\lambda$ of a complex versal deformation of the form (\ref{avd})--(\ref{e8vd}) to the  unordered set of critical values of the corresponding function $f_\lambda$ (taken with their multiplicities), see \cite{Lo0}. 

Now let us take a real morsification of a simple singularity, calculate a virtual morsification related to it (say, by the method of \cite{GZ-1}, \cite{AC}), and apply to it all possible chains of changes (s1)--(s4). We will obtain the complete list of virtual morsifications related to arbitrary non-discriminant real morsifications of the initial singularity. If we apply only the changes not including the jumps of critical values over 0, then we obtain such a list related to morsifications from the same component of the complement of the discriminant. In particular, we have the following statement.

\begin{proposition}
\label{true}
Let $\R^\mu$ be the parameter space of one of versal deformations $($\ref{avd}$)$--$($\ref{e8vd}$)$ of simple singularities. If two strict morsifications from different connected components of $\R^\mu \setminus \Sigma$ are related to one and the same virtual morsification, then the lists of all virtual morsifications related to real morsifications from these components completely coincide. \hfill $\Box$
\end{proposition}

\subsection{Equivalence classes of components of $\R^\mu\setminus \Sigma$ and their invariants}

By Proposition \ref{true} the components of $\R^\mu \setminus \Sigma$ split into {\em equivalence classes}, within any of which the lists of related virtual morsifications do coincide, while such lists of virtual morsifications for different equivalence classes do not have common elements. Such an equivalence class can consist of more than one element: e.g. two different components for singularities $D_{2k}^-$  mentioned in item II of Proposition \ref{d2kmmm} are equivalent in this sense. Moreover, in the case $k=2$ a third component joins them: the one mentioned in item III   of the same proposition (with no ovals). However, we will show below that in the case of $E_\mu$ singularities all these equivalence classes are singletons.

These equivalence classes of components of $\R^\mu \setminus \Sigma$ have two basic invariants. One of them is the number of different virtual morsifications in the corresponding list. The second invariant of the equivalence class of the component containing a point $\lambda \in \R^\mu \setminus \Sigma$ is equal to the difference between the number of real critical values of $f_\lambda$ with negative critical value and even Morse index and the number of real critical points with negative value and odd Morse index. This invariant can be defined also as \ $\chi(W(\lambda))-\chi(W(\lambda_+))$ where $\chi$ is the Euler characteristic and $f_{\lambda_+} \equiv f +1$.

Both our procedures of consecutive formal changes of virtual morsifications (including or not the jumps of critical values over 0)
are realized by Fortran programs, see \cite{pro2}, \cite{APLT}, \cite{fuchs}.
These programs were applied to singularities \ $+E_6, E_7$  and $E_8$, the results of their work imply the following statements.

\begin{proposition}
\label{prog}
In the case $+E_6$ $($respectively,  $E_7$,  $E_8)$ the set of equivalence classes of components of $\R^\mu \setminus \Sigma$  consists of exactly 5 $($respectively, 10, 10$)$ elements. All these equivalence classes are separated by the pair of basic invariants. The 5 $($respectively, 10, 10$)$ morsifications constructed in \S \ref{e6reaa} $($respectively, \S \ref{e7reaa}, \S \ref{e8reaa}$)$ belong to components from different equivalence classes.
\end{proposition}

\begin{proof}
The first program states that there are exactly 456 different virtual morsifications of singularity $+E_6$. The second program was applied separately to the five morsifications constructed in \S \ref{e6reaa} and realizing topological types of sets $W(\lambda)$ shown in the five pictures of Fig.~\ref{e61}. The second basic invariant takes values 0, 1, 2, 3, and $-1$ on these components, hence they belong to different equivalence classes. The program states that these components of $\R^6 \setminus \Sigma$ consist of real morsifications related respectively with 208, 138, 74, 18 and 18 different virtual morsifications. 
 The sum of these numbers is equal to 456, hence all the components of $\R^6 \setminus \Sigma$ belong to equivalence classes of these five components. 

In a similar way, there are exactly 8648 different virtual morsifications of $E_7$ singularity. Five components of $\R^7 \setminus \Sigma$ realized in \S \ref{e7reaa} and representing topological types shown in Fig.~\ref{e71} provide respectively 2040, 1292, 632, 144 and 216 different virtual morsifications, in total exactly a half of 8648. Five components representing the pictures symmetric to ones from Fig.~\ref{e71} have the same values of the first basic invariant but different values of the second one. Namely, the sum of values of the second invariant for any two such symmetric components is equal to $-1$.

Also, there are 51468 different virtual morsifications of $E_8$ singularity. Six components of $\R^8 \setminus \Sigma$ represented by the six functions from \S~\ref{e8reaa} contain real morsifications related respectively to 16600, 8286, 5608, 2652, 576 and 624 virtual morsifications. The corresponding values of the second invariant are equal respectively to $0, 1, 2, 3, 4,$ and $0$. The symmetry mentioned in Proposition \ref{e8pro} takes the four middle components to  components with equal values of the first invariant but opposite values of the second one, so they belong to different equivalence classes. The sum of values of the first invariant over all ten classes is equal to 51468. \end{proof}

\begin{remark} \rm
Our program can also be applied to non-simple singularities, but in this case there is no guarantee that all chains of surgeries (s1)--(s4) of virtual morsifications can be incarnated by paths in $\R^{\mu}$. Nevertheless, this application can provide various useful information, e.g. the conjectures on the existence of morsifications with certain topological properties (which can be further checked by other methods) or strict proof of the absence of morsifications with some such properties, see \cite{APLT}, \cite{fuchs}. 
The answers for $D_\mu$ series were also guessed with the help of this program.

In addition, the number of virtual morsifications produced by our second program starting from any real morsification $f_\lambda$ always is an invariant of the corresponding component of the complement of the discriminant variety.
\end{remark}

\subsection{Triviality of equivalence classes of components  for singularities $E_\mu$} 

\begin{theorem}
\label{the9}
 Each of the 25 equivalence classes mentioned in Proposition \ref{prog} consists of exactly one element.
\end{theorem}

 This theorem and Proposition \ref{prog} imply  Propositions \ref{e6pro}, \ref{e7pro}, and \ref{e8pro} and the part of Theorem  \ref{mthm2} concerning $E_\mu$ singularities. Indeed, they state that the numbers of different components of \ $\R^\mu \setminus \Sigma$ \ in  cases  \ $E_\mu$ \ are equal to the numbers of topological types of sets $W(\lambda)$  realized in \S \ref{realiz}.

 Theorem \ref{mthm} is a corollary of these results. \smallskip

\noindent
{\it Proof of Theorem \ref{the9} for singularity \ $E_8.$} 
In paragraphs {\bf I, II, \dots, VI} below we consider the components of $\R^8 \setminus \Sigma$  
containing the perturbations of $f$ constructed in \S \ref{e8reaa} and realizing the topological types 
 indicated respectively  in the fifth, fourth, third, second, first and sixth pictures of Fig.~\ref{e81}. For any of these components we prove that it is the unique representative of its equivalence class.

\begin{definition} \rm 
The {\it basic perturbation} $f_{\lambda_0}$, $\lambda_0 \in \R^\mu \setminus \Sigma$, of the standard singularity of type $+E_6$ (respectively, $E_7$, $E_8$)  is the polynomial (\ref{e6vd}) (respectively, (\ref{e7vd}), (\ref{e8vd})) 
constructed in \S \ref{e6reaa} (respectively, \ref{e7reaa}, \ref{e8reaa}) and  realizing the fourth (respectively, fourth,
 fifth) picture of Fig.~\ref{e61} (respectively, \ref{e71}, \ref{e81}). 
The {\it basic component} of $\R^\mu \setminus \Sigma$ is the 
component containing the  basic perturbation.
\end{definition}

{\bf I.} Basic perturbation $f_{\lambda_0}$ of the singularity of type $E_8$  has four minima with negative critical values and four saddlepoints with positive ones. It  can be obtained from the perturbation shown in Fig.~\ref{e8proof} (left) by adding a very small generic positive function, keeping all critical values at minima negative but making all eight critical values different. 
All the possible surgeries of functions (see \S \ref{adad}) within its component of $\R^8 \setminus \Sigma$ are the changes of orders of critical values only (not changing their signs), therefore each function from this component has eight real critical points of these types.  (Notice that the value 576 of the first basic invariant in this case is nothing else than $(4!)^2$.) Any component of $\R^8 \setminus \Sigma$ from the  equivalence class of the basic component  consists only of functions $f_\lambda$ satisfying these conditions on the critical points and values.
The complex conjugation acts trivially on the homology group of the Milnor fiber  $V_\lambda = \{(x, y, z) \in  \C^3| f_\lambda(x, y)+z^2=0\}$ for  any polynomial $f_\lambda$ of the form (\ref{e8vd}) satisfying these conditions. By Theorem \ref{1.7.2.thm} all such polynomials belong to a single component of $\R^8 \setminus \Sigma$ (corresponding to the conjugacy class of unity), so the equivalence class of the basic component of $\R^8 \setminus \Sigma$ consists of one element. 
\unitlength 0.8mm
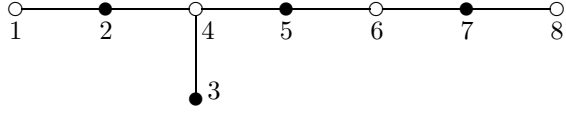
\begin{figure}
\begin{center}
\begin{picture}(90,15)
\put(0,15){\circle{2}} 
\put(15,15){\circle*{2}} 
\put(30,15){\circle{2}} 
\put(45,15){\circle*{2}} 
\put(60,15){\circle{2}} 
\put(75,15){\circle*{2}} 
\put(90,15){\circle{2}} 
\put(30,0){\circle*{2}}
\put(1,15){\line(1,0){13}}  
\put(16,15){\line(1,0){13}} 
\put(31,15){\line(1,0){13}} 
\put(46,15){\line(1,0){13}} 
\put(61,15){\line(1,0){13}} 
\put(76,15){\line(1,0){13}} 
\put(30,14){\line(0,-1){13}}
\put(-1,10){$1$} 
\put(14,10){$2$} 
\put(31,10){$4$} 
\put(44,10){$5$} 
\put(59,10){$6$} 
\put(74,10){$7$} 
\put(89,10){$8$} 
\put(32,0){$3$}
\end{picture}
\end{center}
\caption{Dynkin graph for a basic morsification of the class $E_8$}
\label{Dd}
 \end{figure}

{\bf II.} By \cite{AC}, \cite{GZ-1},  the  Dynkin graph encoding the intersection matrix of  the standard system of (suitably oriented) vanishing cycles of function  $f_{\lambda_0}(x, y) + z^2: \C^3 \to \C$  is as shown in Fig.~\ref{Dd}; white (respectively, black) circles in it correspond to cycles vanishing in minima (respectively, in saddlepoints). Let us number the corresponding critical points of $f_{\lambda_0}$ as shown in this picture. By Theorem \ref{lotriv} the basic component of $\R^\mu \setminus \Sigma$  is contractible, hence this numbering can be uniquely extended to a numbering of critical points of all polynomials $f_\lambda$ from this component depending continuously on $\lambda$. The corresponding labelled Dynkin graphs will be then isomorphic to this one for all $\lambda$ from the basic component such that all critical values of $f_\lambda$ are different.
 
The {\em ordered Looijenga map} from the basic component to $\R^8$ (sending any point $\lambda$ to the set of values of  critical points of $f_\lambda$ ordered in this way) is a diffeomorphism of this component to an octant which is a product of four coordinate intervals $(-\infty,0)$ and four intervals $(0,+\infty)$. The boundary of this component in the parameter space of deformation (\ref{e8vd}) contains exactly eight smooth 7-dimensional open strata, on any of which one of the critical values of corresponding functions becomes equal to 0.
The ordered Looijenga map can be extended to these strata and is non-singular at their points, establishing a diffeomorphism of each of these strata to the product of some seven intervals. In particular, each of these strata is path-connected and separates the basic component of \ $\R^8 \setminus \Sigma$ \ from a single other component. (As we will see later, all these eight neighbour components actually coincide).

Consider the stratum of this kind, on which the critical values of the 7-th critical points of corresponding functions from the basic component become equal to zero. Denote by {\large $C_1$} the component  of $\R^8 \setminus \Sigma$ separated by this stratum from the basic one. The two basic invariants of {\large $C_1$}  are equal to 2652 and 3. Let $\lambda$ be a generic point of our stratum, then subtracting any sufficiently small positive constant function from $f_{\lambda}$ we obtain a function $f_{\lambda'} \in \mbox{\large $C_1$}$ such that 
\begin{itemize}
\item[(i)]
all critical points of $f_{\lambda'}$ are real: these are four minima and four saddlepoints; 
\item[(ii)] all critical values are different;
\item[(iii)] 
the critical values at all saddlepoints are greater than all values at minima; 
\item[(iv)]
the (non-labelled) Dynkin graph of the intersection matrix of the standard system of vanishing cycles is isomorphic to the one from Fig.~\ref{Dd}; 
\item[(v)]
exactly one saddlepoint has a negative critical value: namely, the one corresponding to the second vertex from the end of the longest tail of this graph. 
\end{itemize}
All these properties are formulated in terms of the unique virtual morsification related to $f_{\lambda'}$, hence any component of $\R^8 \setminus \Sigma$ equivalent to {\large $C_1$} also contains a morsification $f_{\tilde \lambda}$ with these properties. Let $c$ be the only negative critical value of $f_{\tilde \lambda}$ at a saddlepoint.
 Adding to $f_{\tilde \lambda}$  constant functions  from the segment $[0,-c+\varepsilon]$ we move  $f_{\tilde \lambda}$ to a component equivalent (and hence equal) to the basic component through the above considered boundary stratum. Therefore $f_{\tilde \lambda}$ also belongs to component {\large $C_1$}. 
So, the equivalence class of  {\large $C_1$}    consists of a single element. 

{\bf III.} The boundary of the basic component contains also ${4 \choose 2}=6$ open strata of dimension 6, any of which consists of functions $f_\lambda$ having exactly two saddlepoints with critical value 0. These strata are smooth and diffeomorphic to $\R^6$.
Consider such a 6-dimensional corner stratum, consisting of functions $f_\lambda$ with the vanishing critical values at the critical points corresponding to  edges 7 and 5 in Fig.~\ref{Dd}.  In a neighbourhood of any  point of this stratum the discriminant variety is ambient diffeomorphic to a pair of intersecting hyperplanes in $\R^8$. The set $\R^8 \setminus \Sigma$ consists in this neighbourhood of four local components, one of which belongs to the basic global component, two others belong to some global components of $\R^8 \setminus \Sigma$ from the equivalence class of {\large $C_1$}, and hence equal to {\large $C_1$} (and to one another). Denote by {\large $C_2$} the global component of $\R^8 \setminus \Sigma$  containing the fourth local component. It belongs to the equivalence class characterized by values 5608 and \ $2$ \ of basic invariants. Its points $f_\lambda$ lying close to our stratum satisfy all conditions (i)--(iv) from the list  of the previous paragraph {\bf II}, and also the modified condition (v) according to which exactly {\em two} saddlepoints, namely the ones corresponding to the second and the fourth vertices from the end of the longest tail of the graph, have negative critical values. 

Again, any component {\large $\tilde C_2$}  of \ $\R^8 \setminus \Sigma$ \ which is equivalent to {\large $C_2$} contains a function  satisfying all these conditions. Using the properness of the Looijenga map, we can move this function within the same component to a function $f_\lambda$ satisfying all these conditions and having equal critical values at two saddlepoints at which these values are negative. Subtracting the constant functions from $f_\lambda$, we enter the basic component from component {\large $\tilde C_2$} through the uniquely defined 6-dimensional stratum of the boundary of the basic component; therefore {\large $\tilde C_2 = C_2$}.

{\bf IV and V.} Analogous arguments concerning the neighborhoods of certain 5- and 4-dimensional smooth corner strata of the boundary of the basic component prove that equivalence classes of components of \ $\R^8 \setminus \Sigma$ \ characterized by the values $(8286,+1)$ and $(16600,0)$ of basic invariants consist of single elements. Denote these components respectively by {\large $C_3$} and {\large $C_4$}.

{\bf VI.} Denote  by {\large $C_5$} the component of $\R^8 \setminus \Sigma$  realized in \S~\ref{e8reaa} with the help of Fig.~\ref{e8proof} (right). Its basic invariants are equal to $624$ and $0$. The proof of its uniqueness in its equivalence class of components
 is based on the consideration of a more complicated stratum of the boundary of the basic component. This stratum is approached from the basic component when the critical values at critical points No. 2, 3, 4 and 5 in Fig.~\ref{Dd} tend to zero. (These points arise from the three vertices and an interior point of the triangle  bounded by the zero set  in Fig.~\ref{e8proof} (left).) This stratum is diffeomorphic to $\R^4$: the global parameters along it are the critical values at the remaining four critical points. For any point $\lambda$ of this stratum the function $f_\lambda$ has a singularity of type $D_4^-$.  An arbitrary 4-dimensional transversal slice of this stratum in $\R^8$ at such a point is the parameter space of a versal deformation of this singularity. According to Proposition \ref{d2kmmm}, the intersection of this slice with $\R^8 \setminus \Sigma$ consists of seven local components. One of them belongs to the basic global component, one to {\large $C_1$}, three to {\large $C_2$},  one to {\large $C_3$} and one to {\large $C_5$}. In particular the equivalence class of {\large $C_5$} is  represented in a neighbourhood of our stratum by a single local component (and a fortiori by a single global component of $\R^8 \setminus \Sigma$). 

\begin{lemma}  Any component of $\R^8 \setminus \Sigma$ equivalent to  {\large $C_5$} approaches this stratum.
\label{le00}
\end{lemma}

\noindent{\it Proof of Lemma \ref{le00}.}
The intersection of our slice with the basic component corresponds to the component of the complement of the discriminant of $D_4^-$ singularity shown in Fig.~\ref{d80} (left) with one oval, and the intersection with {\large $C_5$} corresponds to the component shown in Fig.~\ref{d80} (right), also with one oval. The latter intersection set contains points $\lambda$ arbitrarily close to our stratum and such that the corresponding functions $f_\lambda$ satisfy the following conditions:
\unitlength 1mm
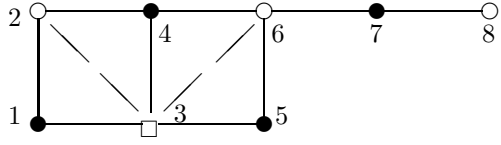
\begin{figure}
\begin{center}
\begin{picture}(90,15)
\put(15,15){\circle{2}} 
\put(15,0){\circle*{2}} 
\put(30,15){\circle*{2}} 
\put(28.5,-1.8){\small $\Box$} 
\put(45,15){\circle{2}} 
\put(45,0){\circle*{2}} 
\put(60,15){\circle*{2}}
\put(75,15){\circle{2}}
\put(16,15){\line(1,0){13}}  
\put(15,14){\line(0,-1){13}} 
\put(16.7,13.3){\line(1,-1){5}} 
\put(28.3,1.7){\line(-1,1){5}}
\put(31.7,1.7){\line(1,1){5}}
\put(43.7,13.3){\line(-1,-1){5}}
\put(46,15){\line(1,0){13}} 
\put(61,15){\line(1,0){13}} 
\put(45,14){\line(0,-1){13}} 
\put(30,14){\line(0,-1){12.5}}
\put(16,0){\line(1,0){12.8}}
\put(44,0){\line(-1,0){13}}
\put(44,15){\line(-1,0){13}}
\put(11,0){$1$} 
\put(11,13){$2$} 
\put(31,10.9){$4$} 
\put(46.5,0){$5$} 
\put(46,10.9){$6$} 
\put(59,10.9){$7$} 
\put(74,10.9){$8$} 
\put(33,0.3){$3$}
\end{picture}
\end{center}
\caption{Dynkin graph for another morsification of the class $E_8$}
\label{Ddd}
 \end{figure}

\begin{itemize}
\item[(i)] all their critical points are real, three of them are minima,  four are saddlepoints, and  one is a maximum;
\item[(ii)] all critical values are different;
\item[(iii)] the Dynkin graph of the intersection matrix of suitably oriented standard vanishing cycles is isomorphic (as a non-labelled graph) to the graph of Fig.~\ref{Ddd} (which is obtained by the method of \cite{GZ-1}, \cite{AC} from Fig.~\ref{e8proof} (right)), where  black circles correspond to saddlepoints and the square to the maximum;
\item[(iv)] only the critical values at the maximum and at the saddlepoint labelled in Fig.~\ref{Ddd} by 7 are positive, and the remaining six critical values are negative;
\item[(v)] all the critical values at saddlepoints are greater than all values at minima and lower than the value at the maximum point.
\end{itemize}
Any component   {\large $\tilde C_5$}  equivalent to {\large $C_5$}  contains a point $\lambda$ such that $f_\lambda$ satisfies these conditions (i)--(v). By the properness of the Looijenga map there exist a path $I:[0,1] \to \R^8$ such that $I(0)=\lambda$, $I([0,1)) \subset \mbox{\large $\tilde C_5$}, $
the critical values at the points No. 1, 3, 4 and 5 (see Fig. \ref{Ddd}) of $F_{I(t)}$ uniformly tend to 0 as $t$  tends to $1$, and remaining critical values stay fixed. The function $f_{I(1)}$ has then a critical point of type $D^-_4$. Let $f_{\tilde \lambda}$ be a 
very small non-discriminant perturbation  of this function deforming this point of type $D_4^-$ in accordance with Fig.~\ref{d80} (left) with one oval. Its virtual morsification is completely determined by that of $f_\lambda$ and coincides with that of the basic morsification $f_{\lambda_0}$. Therefore $\tilde \lambda$ belongs to the basic component, in particular the critical points of $f_{\tilde \lambda}$ are canonically numbered as in paragraph {\bf II} above. The critical values of $f_{\tilde \lambda}$ at critical points No. 2, 3, 4 and 5 (see Fig.~\ref{Dd}) are very close to 0, and remaining four critical values are distant from 0. Therefore point $I(1)$ (approached from component  {\large $\tilde C_5$})  belongs to our stratum of the boundary of the basic component, and hence {\large $\tilde C_5$}= {\large $ C_5$}.  \hfill $\Box$ \medskip

The uniqueness of the components of $\R^8 \setminus \Sigma$ representing equivalence classes with values $(576,-4)$, $(2652, -3)$, $(5608,-2)$ and $(8286,-1)$  of basic invariants can be proved as in paragraphs {\bf I---IV} above, only with the change of all functions $f_\lambda(x, y)$ by $-f_\lambda(-x, -y)$ (this change concerns also the definition of the basic component).  \hfill $\Box$
\medskip

\noindent
{\it Proof of Theorem \ref{the9} in case $E_7$} is almost the same as  for $E_8$ with obvious modifications. 

{\it In  case $E_6$}  we need to take additional care to avoid the ambiguity at entering the basic component from other ones, because the ``longest tail'' of the Dynkin graph $E_6$ is not defined. To do it, we consider on step {\bf II} the 5-dimensional stratum of the boundary of the basic component, which is related to the middle (i.e. invariant under the symmetry of the graph) saddlepoint, and on step {\bf III} the 4-dimensional corner stratum related to the pair of {\em other} two saddlepoints. \hfill $\Box$

\end{document}